%% file: VariableKDE.tex
\documentclass[11pt]{article}
\input{Packages}

\input{MacroFile1}

\begin{document}
\title{Axiomatic Approach to Variable Kernel Density Estimation}
\author{Ilja Klebanov\footnote{Zuse Institute Berlin (ZIB), Takustra\ss e 7, 14195 Berlin, Germany (klebanov@zib.de).}}
\date{\today}
\maketitle
\begin{abstract}
Variable kernel density estimation allows the approximation of a probability
density by the mean of differently stretched and rotated kernels centered at given
sampling points $y_n\in\R^d,\ n=1,\dots,N$. Up to now, the choice of the
corresponding bandwidth matrices $h_n$ has relied mainly on asymptotic arguments, like the minimization of the asymptotic mean integrated squared error (AMISE), which work well for large numbers of sampling points.
However, in practice, one is often confronted with small to moderately sized
sample sets far below the asymptotic regime, which highly restricts the usability
of such methods.

As an alternative to this asymptotic reasoning we suggest an axiomatic approach which guarantees
invariance of the density estimate
under linear transformations of the original density (and
the sampling points) as well as under splitting of the density into several `well-separated' parts. In order to still
ensure proper asymptotic behavior of the estimate, we \emph{postulate} the
typical dependence $h_n\propto N^{-1/(d+4)}$.
Further, we derive a new bandwidths selection rule which satisfies these axioms and performs considerably better than conventional ones in an artificially intricate two-dimensional example as well as in a real life example.
\end{abstract}


\input{sections/Introduction}

\input{sections/FixedPointIteration}
\input{sections/ScalingConditions}
\input{sections/Parzen}

\input{sections/NumericalExperiments}

\input{sections/Conclusion}

%
\appendix
\input{sections/TechnicalDetails}

----------------------------------------------------------------------------------------
\bibliographystyle{abbrv}

\bibliography{myBibliography}

\end{document}

%% file: Packages.tex
\usepackage{amsfonts,amssymb,amsmath,amsthm,amstext,amssymb,mathtools}
\usepackage{caption,subcaption,float,color,graphicx,enumerate}
\usepackage{dsfont}  
\usepackage{fullpage} 
\usepackage{parskip} 
\usepackage{paralist} 

\usepackage{array}  
\newcolumntype{C}[1]{>{\centering\arraybackslash}m{#1}}

\usepackage{tikz}
\usetikzlibrary{bayesnet,er,trees,shapes.symbols,mindmap,arrows,snakes,shapes.misc,shapes.arrows,chains,matrix,positioning,scopes,decorations.pathmorphing,patterns}
\usepackage{tikz-cd} 

\pgfdeclarelayer{bg}    
\pgfsetlayers{bg,main}  

\usepackage{hyperref}

%% file: MacroFile1.tex
\usepackage{mathabx} 
\mathchardef\ordinarycolon\mathcode`\:
\mathcode`\:=\string"8000
\begingroup \catcode`\:=\active
\gdef:{\mathrel{\mathop\ordinarycolon}}
\endgroup

\newcommand*{\Cdot}[1][2.4]{%
  \mathpalette{\CdotAux{#1}}\cdot%
}
\newdimen\CdotAxis
\newcommand*{\CdotAux}[3]{%
  {%
    \settoheight\CdotAxis{$#2\vcenter{}$}%
    \sbox0{%
      \raisebox\CdotAxis{%
        \scalebox{#1}{%
          \raisebox{-\CdotAxis}{%
            $\mathsurround=0pt #2#3$%
          }%
        }%
      }%
    }%
    \dp0=0pt %
    \sbox2{$#2\bullet$}%
    \ifdim\ht2<\ht0 %
      \ht0=\ht2 %
    \fi
    \sbox2{$\mathsurround=0pt #2#3$}%
    \hbox to \wd2{\hss\usebox{0}\hss}%
  }%
}

\DeclarePairedDelimiter\abs{\lvert}{\rvert}%
\DeclarePairedDelimiter\norm{\lVert}{\rVert}%
\makeatletter
\let\oldabs\abs
\def\abs{\@ifstar{\oldabs}{\oldabs*}}
\let\oldnorm\norm
\def\norm{\@ifstar{\oldnorm}{\oldnorm*}}
\makeatother

\definecolor{dunkelgruen}{rgb}{0,0.4,0}

\DeclareMathOperator{\diag}{diag}

\def\R{\mathbb{R}}
\def\N{\mathbb{N}}

\def\V{\mathbb{V}}

\def\Id{{\rm Id}}

\theoremstyle{plain}
\newtheorem{theorem}{Theorem}

\newtheorem{proposition}[theorem]{Proposition}

\newtheorem{example}[theorem]{Example}

\newtheorem{axiom}[theorem]{Axiom}
\newtheorem{remark}[theorem]{Remark}

\def\MSE{{\rm MSE}}

\def\opt{{\rm opt}}
\def\mtm{\mathtt{m}}
\def\mth{\mathtt{h}}
\def\mtg{\mathtt{g}}
\def\hot{\rm h.o.t.}
\def\MM{{\rm MM}}

\def\cY{\mathcal{Y}}
\def\B{\mathbb{B}}

\def\gldr{\mathrm{GL}(d,\R)}
\def\tr{{\rm tr}}

%% file: sections/Introduction.tex
\section{Introduction}
\label{section:Introduction}

The classical density estimation problem is to recover a probability density
$\rho$ from independent and identically distributed samples from that density, $y_1,\dots,y_N\stackrel{\rm iid}{\sim}\rho$.
A widely used nonparametric technique is kernel density estimation (KDE), see
e.g. the classical works
\cite{rosenblatt1956remarks,zbMATH03188880,silverman1986density} or
\cite{wand1994kernel, scott2015multivariate, botev2010kernel} for more recent
surveys, which approximates $\rho$ by the mean of so-called kernels centered at
the sample points $y_n$,
\begin{equation}
\label{equ:kde}
\hat{\rho}(x)
=
\frac{1}{N}\sum_{n=1}^{N} K_h\left(x-y_n\right)
=
\frac{1}{Nh^d}\sum_{n=1}^{N} K\left(h^{-1}(x-y_n)\right),
\end{equation}
where $h>0$ is the bandwidth of the kernel function $K\colon \R^d\to\R$.
From now on we will assume that $\rho$ lies in the space $C^2\cap
L^2(\R^d)$ and the kernel $K\in C^2(\R^d)$ is a radially symmetric probability
density function, i.e.
\begin{equation}
\label{equ:radialKernel}
\norm{K}_{L^1(\R^d)} = 1,
\qquad
K(x) = \gamma(\|x\|_2^2),
\end{equation}
for some function $\gamma\colon \R_{\ge 0}\to \R_{\ge 0}$.
A lot of effort has been put into the `optimal' choice of the bandwidth $h$, see e.g. \cite{jones1996brief,loader1999bandwidth} -- choosing it too large or too small results in oversmoothing or undersmoothing, as visualized in Figure \ref{fig:OversmoothingUndersmoothing}.
\begin{figure}[H]
        \centering
        \begin{subfigure}[b]{0.48\textwidth}
            \centering
			\includegraphics[width=\textwidth]{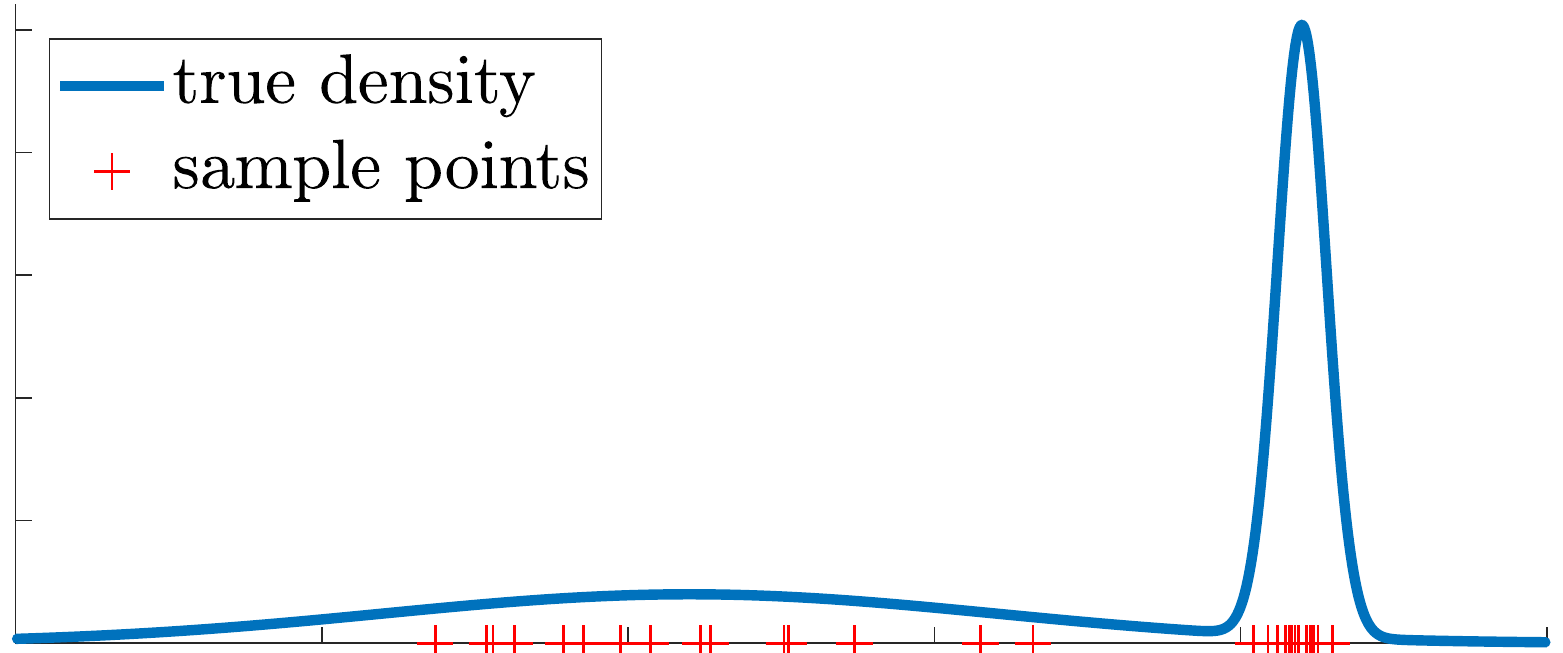}
            \caption{True density with 30 sampling points}
        \end{subfigure}
        \hfill
        \begin{subfigure}[b]{0.48\textwidth}
            \centering
			\includegraphics[width=\textwidth]{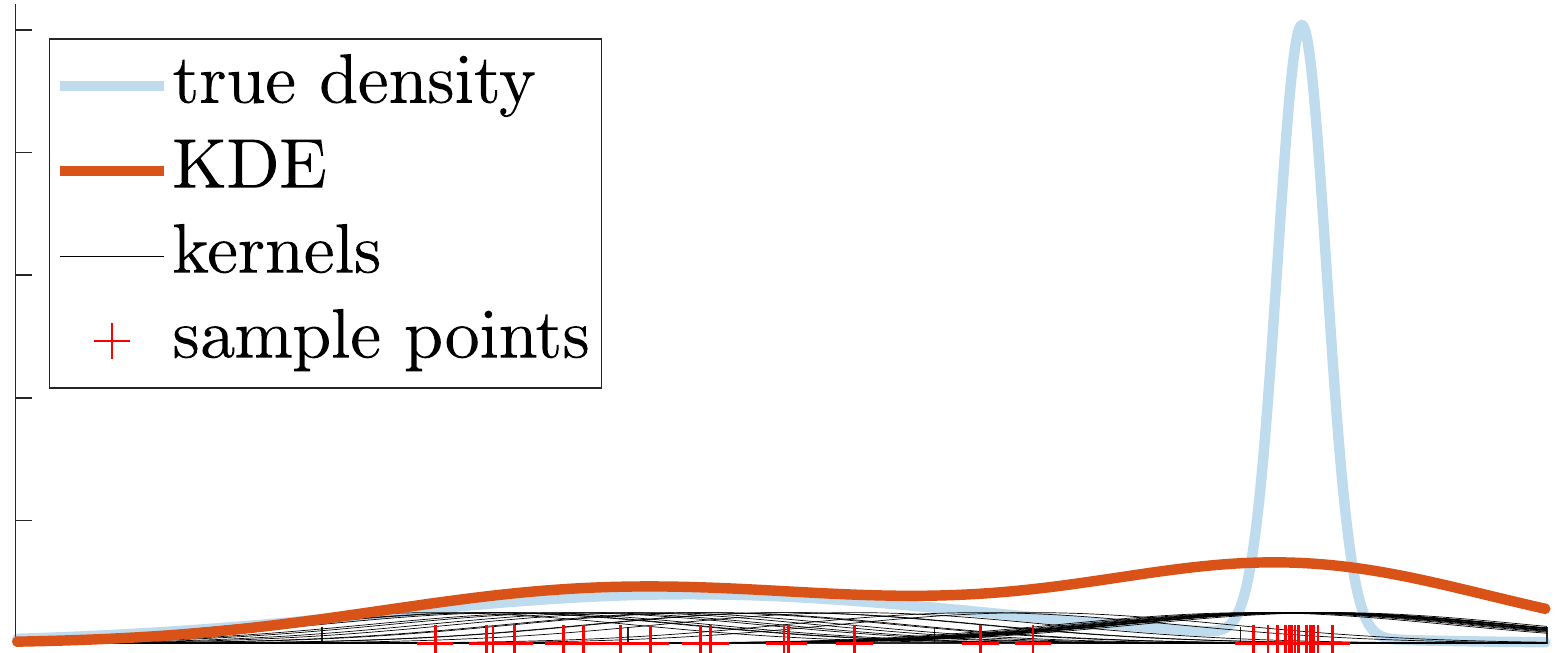}
			\caption{Flat kernels oversmooth the `right' part}
		\end{subfigure}
		\vspace{0.2cm}
		\vfill
        \begin{subfigure}[b]{0.48\textwidth}
            \centering
			\includegraphics[width=\textwidth]{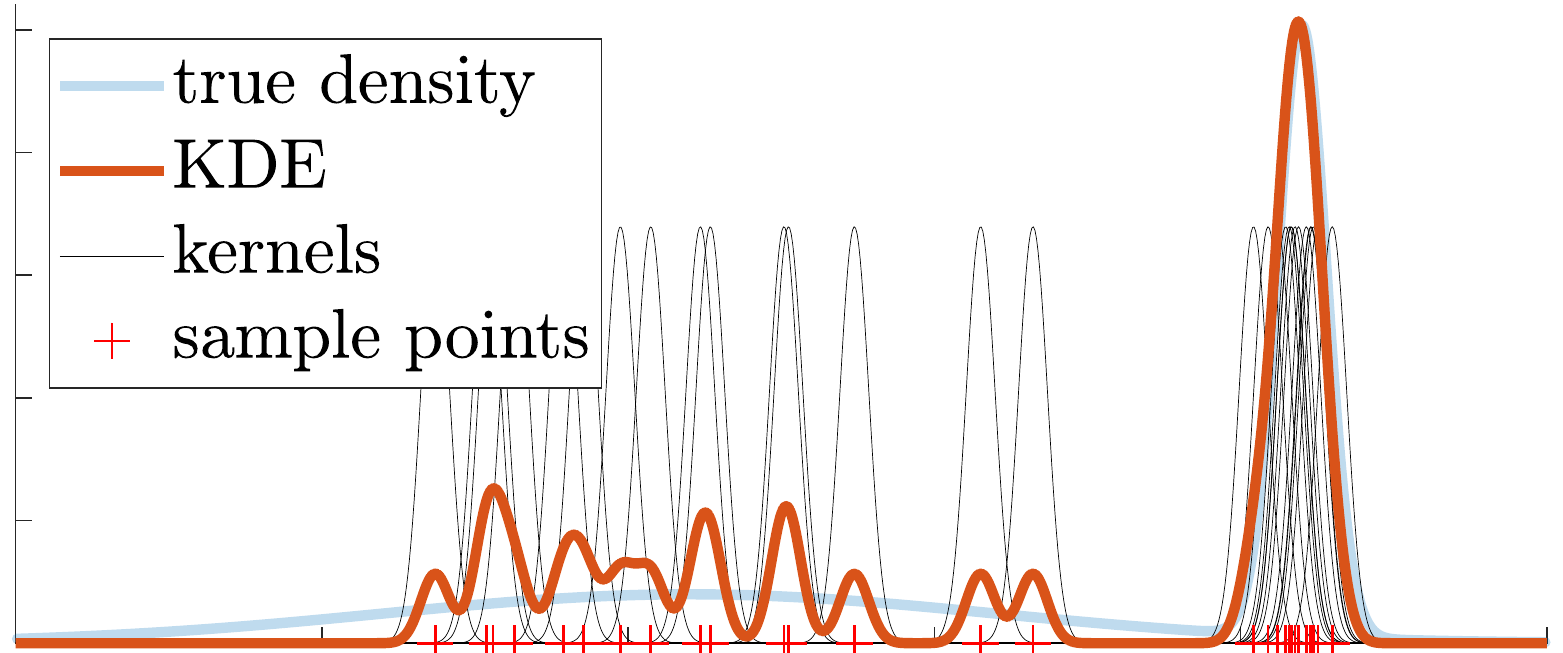}
			\caption{Peaked kernels undersmooth the `left' part}
        \end{subfigure}
        \hfill
        \begin{subfigure}[b]{0.48\textwidth}
            \centering
			\includegraphics[width=\textwidth]{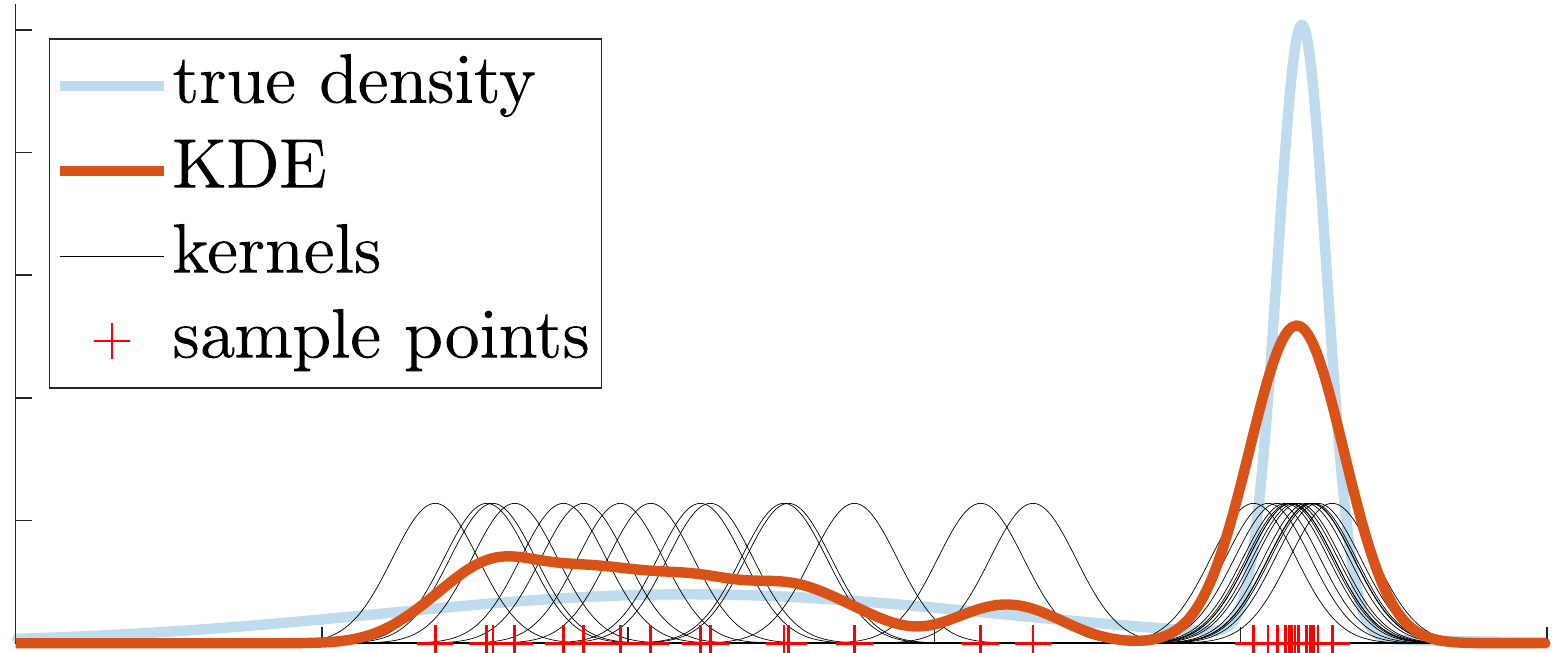}
            \caption{Optimal kernels still over-/undersmooth}
        \end{subfigure}
		\vspace{0.2cm}
		\vfill
        \begin{subfigure}[b]{0.48\textwidth}
            \centering
			\includegraphics[width=\textwidth]{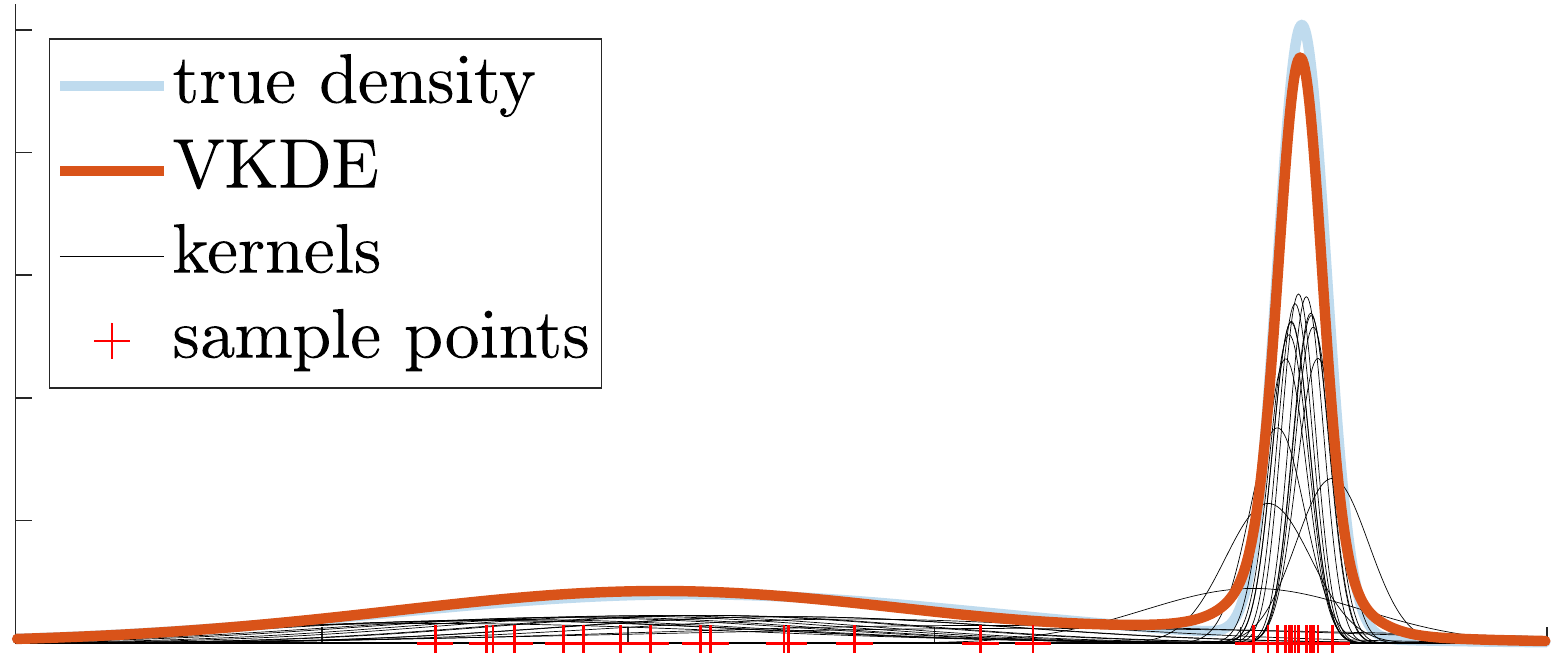}
			\caption{VKDE using \eqref{equ:hLaw} with $\beta = 1$}
        \end{subfigure}        	
        \caption{Choosing a suitable bandwidth for the `flat' part oversimple the `peaked' part. Choosing a suitable bandwidth for the `peaked' part undersmoothes the `flat' part. A trade-off between the two also yields an unsatisfactory density estimate. This dilemma of KDE can be overcome by VKDE, where the bandwidth is adapted locally. The kernels (divided by the factor 5 for illustration purposes) are plotted in black. The proportionality constant in \eqref{equ:hLaw} was chosen manually.
		}
        \label{fig:OversmoothingUndersmoothing}
\end{figure}
Optimality is usually measured by the \emph{mean integrated squared error} (MISE) or its asymptotic approximation (AMISE),
see \cite{silverman1986density,scott2015multivariate}.
A remarkable observation is the type of dependence of the optimal bandwidth $h_\opt$ on the number $N$ of sample points (\cite[equation (3.21)]{silverman1986density}),
\begin{equation}
\label{equ:hDependsOnN}
h_\opt\propto N^{-1/(d+4)},
\end{equation}
which appears counterintuitive (e.g. in the univariate case $d=1$ one would
expect that doubling the number of points corresponds to half as wide kernels).

But even an optimally chosen bandwidth can still cause oversmoothing in regions of high point density and cause peaked behavior of $\hat{\rho}$ in regions where only few points lie, see Figure \ref{fig:OversmoothingUndersmoothing} (d).
\emph{Variable kernel density estimation} (VKDE) tries to overcome this downside by adapting $h$ \emph{locally} (Figure \ref{fig:OversmoothingUndersmoothing} (e)). Roughly speaking, there are two possibilities to do so:
sample-point estimation, employing a different bandwidth $h_n$ for each data point $y_n$, and balloon estimation, for which the bandwidth $h(x)$ varies with the estimation location $x\in\R^d$, see the discussion and graphic illustration in \cite{jones1990variable}.
We will concentrate on sample-point estimators,
\begin{equation}
\label{equ:rhoV}
\hat{\rho}_\mth(x)
=
\frac{1}{N}\sum_{n=1}^{N} |\det h_n|^{-1} K\left(h_n^{-1}(x-y_n)\right),
\qquad
\mth = (h_1,\dots,h_N),
\end{equation}
since, in contrast to balloon estimators, they result in probability density functions by construction.
Here, we also generalized the standard definition to matrix-valued bandwidths
$h_n\in\mathrm{GL} (d,\R)$, such that each kernel can be stretched and rotated in space.

Let us first deal with the case of scalar bandwidths $h_n>0$ before discussing the matrix-valued case (here, the coefficients $\abs{\det h_n}^{-1}$ have to be replaced by $h_n^{-d}$).
Since we prefer peaked kernels in areas of high density and flat kernels in areas of low density,
a dependence of the form 
\begin{equation}
\label{equ:hLaw}
h_n \propto\, N^{-1/(d+4)}\, \rho(y_n)^{-\beta},
\end{equation}
where $\beta >0$ is the so-called sensitivity parameter, appears natural.
While \cite{zbMATH03800791} argues that $\beta = 1/2$ should be used independent of the dimension, \cite{zbMATH03591205} suggest $\beta = 1/d$, which guarantees consistency of the sample-point estimator $\hat\rho_V$ under scaling
-- if the density and the sample points are both scaled in space by a factor $\alpha>0$, the estimate is scaled correspondingly:
\begin{equation}
\label{equ:simpleScalingEstimator}
\rho'(x) = \alpha^d\rho(\alpha x),
\quad
y_n' = \alpha^{-1}y_n
\qquad
\text{implies}
\qquad
\hat \rho_{\mth'}'(x) = \alpha^d\, \hat \rho_\mth(\alpha x).
\end{equation}
However, both choices are inconsistent if the scaling is performed by a matrix
$A\in\gldr$. One requires more sophisticated rules than \eqref{equ:hLaw} in
order to guarantee the more general scaling condition
\begin{equation}
\label{equ:GeneralScalingEstimator}
\rho'(x) = \abs{\det A}\,\rho(A x),
\quad
y_n' = A^{-1}y_n
\qquad
\text{implies}
\qquad
\hat \rho_{\mth'}'(x) = \abs{\det A}\, \hat \rho_\mth(A x),
\end{equation}
see the discussion in Section \ref{section:ScalingConditionsAndChoice}.

Earlier, Parzen (\cite[equation (4.15)]{zbMATH03188880}) derived the following law for $h_n$ in the univariate case by minimizing the minimal squared error (MSE):
\begin{equation}
\label{equ:ParzenBandwidth}
h_n
=
\left(
\frac{C(K)\, \rho(y_n)}{N\, \rho''(y_n)^2}
\right)^{\frac{1}{5}},
\quad
C(K) := \frac{\int K^2(t)\, \mathrm dt}{\left(\int t^2K^2(t)\, \mathrm dt\right)^2},
\end{equation}
where we again observe the dependence $h_n\propto N^{-1/5}$ as in \eqref{equ:hDependsOnN}.
Since minimizing the MSE locally asymptotically corresponds to minimizing the
MISE, see the discussion in Section \ref{section:Parzen} or in \cite[Chapter 6.6]{scott2015multivariate}, this formula is of great interest for VKDE.
However, Parzen did not have in mind the application to \emph{variable} KDE and, as discussed in Sections \ref{section:Parzen} and \ref{section:NumericalExperiments}, the law \eqref{equ:ParzenBandwidth} is difficult to generalize to higher dimensions and can perform poorly for small sample sizes.

\subsection{Axiomatic Approach to Bandwidth Selection}
While asymptotically optimal bandwidths selectors provide good results for large sample sizes, they are usually not the appropriate tool if the number of samples is small or only moderately large. If the sample size is far below the asymptotic regime, an alternative approach appears necessary.

We suggest to base the selection of the bandwidths on certain invariance axioms.
Apart from invariance of the density estimate under shifting of the original density (and the sample points), which is fulfilled by most KDE and VKDE estimates, and the scaling invariance \eqref{equ:GeneralScalingEstimator}, we introduce invariance of the estimator under `splitting' of the original density (and the corresponding sample points) into well-separated parts. This condition is an entirely new concept,
which we will shortly sketch here and discuss in more detail in Section \ref{section:ScalingConditions} (in particular Theorem \ref{theorem:InvariancePropertiesVKDE}(ii), Remark \ref{rem:InvariancePropertiesVKDE} and Figure \ref{fig:TwoDensitiesPlot}):

If a density is a convex combination of two densities $\rho^{(1)},\, \rho^{(2)}$
with disjoint and far-apart supports $\Omega_1,\, \Omega_2\subset\R^d$, its
density estimate $\hat\rho_\mth$ based on the sampling $\cY =
(y_1,\dots,y_N)\stackrel{\rm iid}{\sim}\rho$ should be approximately the
(similar) convex combination of the density estimates
$\hat\rho_{\mth^{(1)}}^{(1)},\, \hat\rho_{\mth^{(2)}}^{(2)}$ based on the same
sampling points in the respective domains, $\cY\cap\Omega_1,\, \cY\cap\Omega_2$:
\[
\rho = \alpha\rho^{(1)} + (1-\alpha)\rho^{(2)},\ \alpha\in[0,1]
\quad
\text{should imply}
\quad
\hat\rho_\mth \approx \alpha\hat\rho_{\mth^{(1)}}^{(1)} + (1-\alpha)\hat\rho_{\mth^{(2)}}^{(2)}.
\]
The approximation sign becomes an equality if we let the distance between the two domains $\Omega_1$ and $\Omega_2$ converge to infinity.
Relying on the theory of adaptive convolutions and the concept of the local variation of a function, we derive a new bandwidth selection rule which fulfills the proposed axioms and shows superior performance in several examples.


The paper is structured as follows.
Section \ref{section:FixedPointIteration} addresses the implementation of bandwidths selection rules like \eqref{equ:hLaw} and \eqref{equ:ParzenBandwidth} (in practice, the true density is, of course, not accessible).
In Section \ref{Section:AdaptiveConvo}, we give a short overview on adaptive convolutions, which inspires both, the invariance axioms introduced in Section \ref{section:ScalingConditions} as well as the bandwidth selection rule analyzed in Section \ref{section:BandwidthsChoice}.
In Section \ref{section:Parzen} we revisit Parzen's law \eqref{equ:ParzenBandwidth} in an attempt to generalize it to the multivariate case.
A comparison of the different VKDE methods is illustrated by two examples with artificial as well as real life data in Section \ref{section:NumericalExperiments}.
Section \ref{section:Conclusion} gives a short conclusion, while Appendix
\ref{section:technicalDetails} discusses some computational details in the
case of Gaussian kernels.

%% file: sections/FixedPointIteration.tex
\section{Practical Realizations of the Laws \eqref{equ:hLaw},
\eqref{equ:ParzenBandwidth} and Similar}
\label{section:FixedPointIteration}

For theoretical considerations it is common to choose $h_n$ in dependence of $\rho$, $y_n$ and $N$ (and possibly of derivatives of $\rho$ as in \eqref{equ:ParzenBandwidth}) in order to show invariance properties or optimality in some sense. Of course, in practice, the true and unknown density $\rho$ in not accessible and one is forced to switch to pilot estimates (e.g. kernel density estimates of $\rho$ with a fixed bandwidth (\cite{abramson1982arbitrariness})), to asymptotic approximations by using the ($k$th) nearest neighbors of the points $y_n$ (\cite{zbMATH03591205}) or similar.
Surprisingly, the application of a fixed point iteration for the inverse bandwidths $\mth = (h_n)_{n=1,\dots,N}$ has not yet been suggested (to the author's best knowledge), though such a method is strongly related to the \emph{solve-the-equation bandwidth selector}, see e.g. \cite{jones1996brief}.
For a law of the general form
\[
h_n = \Phi_N(\rho,y_n),
\qquad
n=1,\dots,N,
\]
such as \eqref{equ:hLaw} or \eqref{equ:ParzenBandwidth}, and starting with
initial bandwidths $\mth^{(0)} = (h_1^{(0)},\dots,h_N^{(0)})$, we propose the
iteration
\begin{equation}
\label{equ:densityFixedPointIteration}
h_n^{(k+1)} = \Phi_N(\rho_{\mth^{(k)}}, y_n),
\qquad
n=1,\dots,N,\ k\in\N.
\end{equation}
\begin{figure}[H]
\centering
\includegraphics[width=0.6\textwidth]{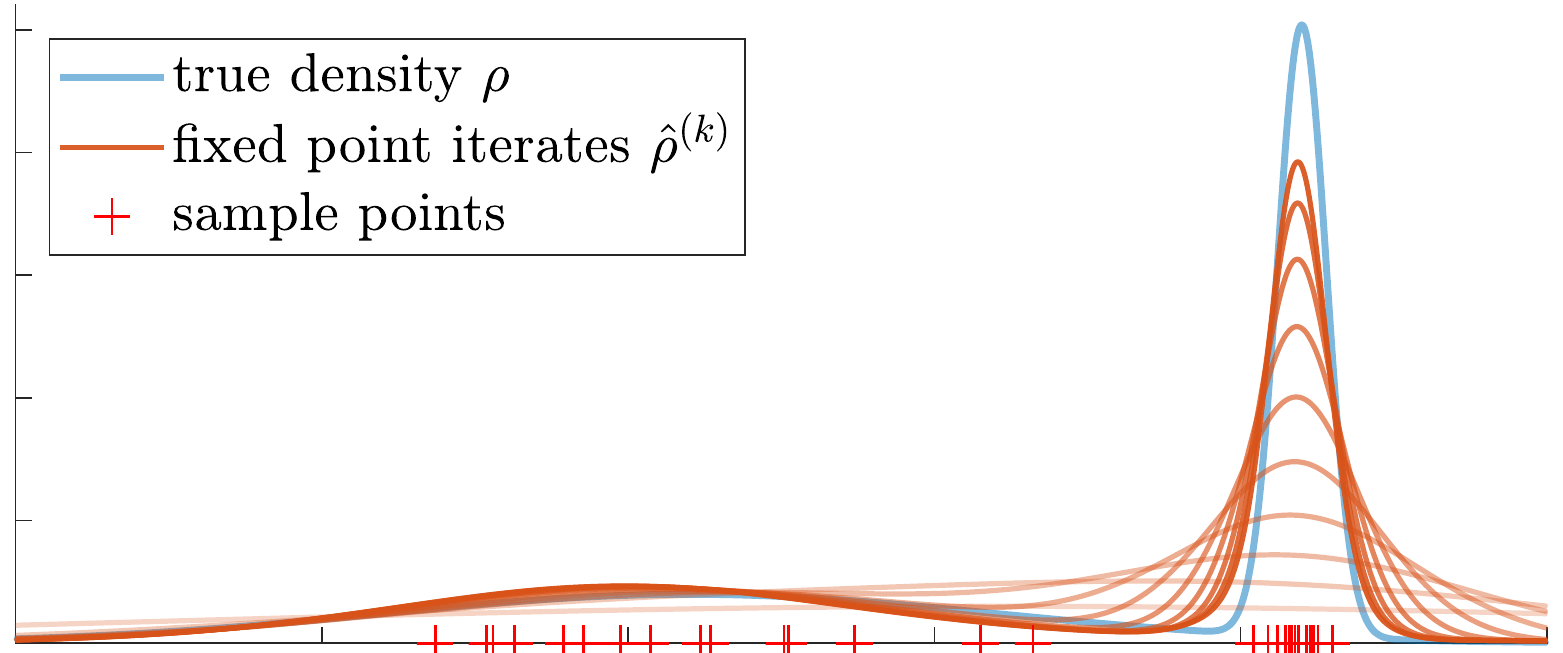}
\caption{The fixed point iteration \eqref{equ:densityFixedPointIteration} for the law \eqref{equ:hLaw} with $\beta = 1$ and $N=30$ sample points $y_n\stackrel{\rm iid}{\sim}\rho$ from the true density $\rho$
started with wide and equal bandwidths of the kernels. Ten iterates are plotted with increasing saturation value. The proportionality constant in \eqref{equ:hLaw} was chosen manually.
}
\label{fig:densityFixedPointIteration}
\end{figure}
As visualized in Figure \ref{fig:densityFixedPointIteration} for the law
\eqref{equ:hLaw}, it performs far better then just a pilot estimate (which
corresponds to the first step of the iteration) and is self-consistent in the
sense that the fixed point $\mth^\ast$ fulfills 
\[
\mth_n^{\ast} =\Phi_N(\rho_{\mth^{\ast}},y_n),
\quad
n=1,\dots,N,
\]
which is similar to the defining property of the solve-the-equation bandwidth selector.\footnote{To be more precise, solve-the-equation bandwidth selectors have the general form $\mth_n^{\ast} =\Phi_N(\rho_{\mtg(\mth^{\ast})},y_n)$, since bandwidths $\mth^\ast$ that are favorable for the estimation of $\rho$ are not necessarily suitable for the estimation of $\Phi_N(\rho,y_n)$, see \cite{jones1996brief}.
We will not deal with this issue here, but it is a promising direction for future research.}
The convergence properties of such fixed point iterations in dependence of the map $\Phi_N$ are still to be analyzed.

%% file: sections/ScalingConditions.tex
\section{Scaling Axioms and Choice of the Bandwidths}
\label{section:ScalingConditionsAndChoice}

In this section, we will introduce certain invariance axioms we want our
sample-point estimator to fulfill.
We will then derive a law for the bandwidths $h_n$ which satisfies these
axioms in Section \ref{section:BandwidthsChoice}.
Apart from requiring invariance under
shifting, we will generalize the simple scaling condition
\eqref{equ:simpleScalingEstimator} from positive factors $\alpha$ to invertible
matrices $A$ as in equation \eqref{equ:GeneralScalingEstimator} and, more
importantly, we will introduce the new argument sketched in the introduction which leads to yet another invariance axiom (see Axiom \ref{cond:ScalingConditions} (I2)).

The invariance axioms we formulate are analogues of the adaptation axioms in
\cite{2018arXiv180500703K} and we will make use of the adaptation function \eqref{equ:choiceMu}
introduced below, therefore the following subsection will be a short overview of
the theory of adaptive convolutions.

\subsection{Adaptive Convolutions}
\label{Section:AdaptiveConvo}

Smoothing a function $f\in W^{2,2}(\R^d)$ by a radially symmetric smoothing
kernel $g\in L^1(\R^d)$, the behavior of which varies strongly in space, often requires the possibility to control the amount of smoothing locally. This can be
realized by replacing the constant smoothing coefficient $\sigma >0$ in the
standard convolution,
\begin{equation}
(f\ast g_\sigma) (x)= \int f(y) \, g_\sigma(x-y)\, \mathrm dy\, ,
\qquad
g_\sigma(x) = \sigma^{-d} g\left(x/\sigma\right),
\end{equation}
by a (possibly matrix-valued) function $\mu:\R^d\to \mathrm{GL} (d,\R)$:
\begin{equation}
\label{equ:firstDef}
(f\ast_{\mu} g) (x):= \int f(y)\, |\det\mu(y)|\, g\big(\mu(y)(x-y)\big)\,
\mathrm dy\, .
\end{equation}
The theoretical framework for such \emph{adaptive convolutions} was developed in \cite{2018arXiv180500703K}, where also an implicit formula for the automatic choice of the so-called
\emph{adaptation function} $\mu$ in dependence of $f$ was derived,
\begin{equation}
\label{equ:choiceMu}
\mu_f^2(x) = \frac{\left( \nabla
f\nabla f^{\intercal} - f\, D^2 f\right) \ast
G_{(\lambda\mu_f)^{-2}(x)}^2}{(2-\lambda^2)\, f^2 \ast
G_{(\lambda\mu_f)^{-2}(x)}^2}(x)\, ,
\end{equation}
where $0<\lambda<\sqrt{2}$ and $G_\Sigma$ denotes the Gaussian function with
mean zero and covariance matrix $\Sigma$. This choice is motivated by certain phase
space transformations as well as the requirement to fulfill the following adaptation axioms, which ensure proper
behavior under shifting and scaling of $f$:
\begin{axiom}[Adaptation Axioms]
\label{cond:adaptation}
Let $\mathcal M = \{\mu\colon \R^d\to\gldr\colon \mu \text{ measurable}\}$.
We say that a mapping
$$
\mtm\colon W^{2,2}(\R^d,\R)\to \mathcal M,
\qquad
f\mapsto \mu_f,
$$
fulfills the \emph{Adaptation Axioms}, if for any $a\in\R^d$,
$\alpha\in\R\setminus\{0\}$, $A\in\gldr$, any parametrized function
$f^{(t)} = \sum_{k=1}^K f_k(\Cdot-a_k^{(t)})$, $t\ge 0,$ with $f_k\in
W^{2,2}(\R^d,\R)$, $a_k^{(t)}\in\R^d$, such that
$\|a_k^{(t)}-a_j^{(t)}\|\xrightarrow{t\to\infty}\infty$ for all $k\neq j$, and
any $x\in\R^d$,
\begin{enumerate}
  \item[(A1)]
  $\displaystyle
  \mu_{f(\Cdot\, - a)}(x) = \mu_f(x-a)$ \hfill (invariance under shifting),
  \item[(A2)]
  $\displaystyle \mu_{\alpha\cdot f} = \mu_f$ \hfill (invariance under scalar multiplication),
  \item[(A3)]
  $\displaystyle \mu_{f(A\cdot\, \Cdot)}^\intercal(x)\, \mu_{f(A\cdot\, \Cdot)}(x)
  =
  A^{\intercal}\, \mu_f^\intercal(A x)\,\mu_f(A x)\, A$
  \hfill (invariance under scaling),
  \item[(A4)]
  $\mu_{f^{(t)}}(x+a_k^{(t)}) \xrightarrow{t\to\infty}\mu_{f_k}(x)$ for all
  $k=1,\dots,K$
  \hfill
  (locality).
\end{enumerate}
\end{axiom}
Apart from these axioms, $\mu_f$ should measure in some sense the local variation of $f$, which is why the choice \eqref{equ:choiceMu} was derived by means of certain phase space transforms, see \cite{2018arXiv180500703K}.
As mentioned above, axioms (A1)--(A3) guarantee the invariance of the adaptive
convolution \eqref{equ:firstDef} under shifting and scaling of $f$. In
addition, if $f = \sum_{k=1}^K f_k$ is the sum of several functions
$f_1,\dots,f_K$ with `far apart' supports, (A4) ensures that it is
smoothed approximately the same way as these functions would have been smoothed
separately, $f\ast_{\mu_f} g \approx \sum_{k=1}^K f_k\ast_{\mu_{f_k}}g$.
These implications are summarized in the following proposition:
\begin{proposition}
\label{prop:AdaptationConditions}
Assuming Adaptation Axioms \ref{cond:adaptation} and adopting that
notation, we have for each  $f\in W^{2,2}(\R^d,\R)$, radially symmetric $g\in
L^p$ and $x\in\R^d$:
\begin{enumerate}[(i)]
  \item $\displaystyle (f(\Cdot\, - a)\ast^p_{\mu_{f(\Cdot\, - a)}} g) (x)=
  (f\ast^p_{\mu_f}g) (x-a)$
  \hfill
  \small (shifted function $\Rightarrow$ shifted convolution)\normalsize,
  \item $\displaystyle (\alpha f)\ast^p_{\mu_{\alpha f}} g =
  \alpha (f\ast^p_{\mu_f}g)$
  \hfill
  \small (stretched function $\Rightarrow$ stretched convolution)\normalsize,
  \item $\displaystyle (f(A\cdot\Cdot)\ast^p_{\mu_{f(A\cdot\Cdot)}} g) (x)=
  (f\ast^p_{\mu_f}g) (Ax)$
  \hfill
  \small (scaled function $\Rightarrow$ scaled convolution)\normalsize.
  \item $(f^{(t)}\ast_{\mu_{f^{(t)}}} g)(x) = \sum_{k=1}^K
  (f_k\ast_{\mu_{f_k}}^p g)(x-a_k^{(t)})$ asymptotically for $t\to\infty$,
  \\
  more precisely:  
  $(f^{(t)}\ast_{\mu_{f^{(t)}}}^p g)(x+a_k^{(t)})
  \xrightarrow{t\to\infty}(f_k\ast_{\mu_{f_k}}g)(x)$
  \hfill
  \small (locality)\normalsize.
\end{enumerate}
\end{proposition}
\begin{proof}
See \cite{2018arXiv180500703K}.
\end{proof}

\begin{proposition}
\label{prop:muFulfillsAdaptationConditions}
The adaptation function $\mu_f$ given by \eqref{equ:choiceMu} fulfills the Adaptation Axioms \ref{cond:adaptation}.
\end{proposition}
\begin{proof}
See \cite{2018arXiv180500703K}.
\end{proof}

\begin{figure}[H]
        \centering
		\includegraphics[width=1\textwidth]{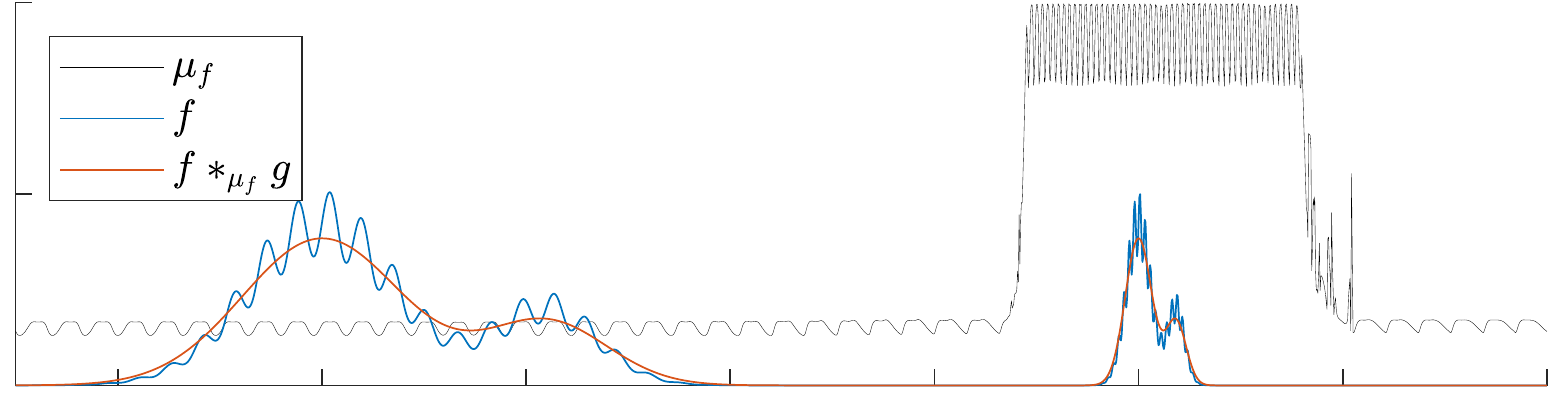}
        \caption{$\mu_f$ given by \eqref{equ:choiceMu} describes locally the variation of $f$. Choosing it as an adaptation function for the convolution of $f$ and $g$ (here, $g$ is a Gaussian function) yields a proper local scaling of $g$ and thereby an adequate smoothing of $f$ everywhere.}
        \label{fig:convolution6und7und8}
\end{figure}

There are at least three reasons why one should consider adaptive convolutions when dealing with VKDE:
\begin{compactitem}
\item
The Adaptation Axioms \ref{cond:adaptation} are a good starting point for the
formulation of our scaling axioms (however, as we will see in Section
\ref{section:ScalingConditions}, these axioms have to be adopted with care
to the VKDE setup).
\item
Due to the properties of the adaptation function $\mu_f$ (see Proposition
\ref{prop:muFulfillsAdaptationConditions}), it is a good starting point for the choice of the bandwidths $h_n$.
\item
Just as the standard KDE \eqref{equ:kde} converges (almost surely) to the (standard) convolution of the density $\rho$ and the kernel $K_h$ by the law of large numbers,
\begin{align*}
\hat{\rho}(x)
&=
\frac{1}{N}\sum_{n=1}^{N} K_h\left(x-y_n\right)
\xrightarrow{N\to\infty}
(\rho\ast K_h) (x),
\intertext{the VKDE \eqref{equ:rhoV} converges to their adaptive convolution
with adaptation function $h^{-1}$,}
\hat{\rho}_V(x)
&=
\frac{1}{N}\sum_{n=1}^{N} \abs{\det h(y_n)}^{-1} K\left(h(y_n)^{-1}(x-y_n)\right)
\xrightarrow{N\to\infty}
(\rho\ast_{h^{-1}} K) (x)
\end{align*}
(in both cases we assumed that $h$ is chosen independently from $N$ and that
$h_n = h(y_n)$ for some function $h\colon \R^n\to\gldr$ in the second case).
Therefore, adaptive convolutions are an important theoretical tool for the
analysis of VKDE.
\end{compactitem}
%
%
%
%

\subsection{Invariance Axioms}
\label{section:ScalingConditions}

While the choice \eqref{equ:hLaw} for $\beta = 1/d$ behaves well under scaling with a factor $\alpha>0$, the property \eqref{equ:simpleScalingEstimator} does not generalize to scaling with arbitrary invertible matrices $A\in\gldr$ as formulated in equation \eqref{equ:GeneralScalingEstimator}.
In order to get proper scaling properties in higher dimensions, we will
therefore formulate axioms analogous to the Adaptation Axioms
\ref{cond:adaptation}, before finding a better law for the bandwidths than \eqref{equ:hLaw}.
Some caution is advised concerning the translation of these axioms to the VKDE setup: the choice $h_n\propto \mu_\rho^{-1}(y_n)$ appears natural, since $\mu_\rho$ describes the local variation of $\rho$. However, $\mu_\rho$ fails to depend on the number of sample points $y_n$ that lie in a certain region, as illustrated by the following example:

\begin{example}
Consider a density of the form
\[
\rho(x) = \frac{1}{3}\rho_1(x) + \frac{2}{3}\rho_2(x),
\qquad
\rho_2(x) = \rho_1(x-a),
\]
where $\rho_1\colon \R\to\R$ is a density with bounded support and the shift $a\in\R$ clearly separates $\rho_1$ and $\rho_2$ in space.
Naturally, there will be roughly twice as many points in the support of $\rho_2$
as in the one of $\rho_1$ and the kernels can be chosen more peaked in the
support of $\rho_2$ (see e.g. the dependence of $h_\opt$ on the number of points
in \eqref{equ:hDependsOnN}). However, choosing $h_n\propto \mu_\rho^{-1}(y_n)$
would force the kernels in the two regions to have similar bandwidths by Adaptation Axiom
\ref{cond:adaptation} (A2)!
\end{example}
In order to account for this crucial difference between adaptive convolutions
and VKDE, we will have to essentially modify Adaptation Axiom
\ref{cond:adaptation} (A2) (even though such a condition might seem rather
artificial for normalized densities).
The new choice relies on the dependence of $h_n$ on the number of sample
points $N$, which we will \emph{assume} to be of the form \eqref{equ:hDependsOnN}. More precisely, we presume
\begin{equation}
\label{equ:mDependsOnN}
\framebox[4.5cm][c]{
$\displaystyle 
h_n \stackrel{!}{=} N^{-1/(d+4)} \Phi(\rho,y_n).
$
}
\end{equation}
\begin{axiom}[Invariance Axioms]
\label{cond:ScalingConditions}
A map $\Phi\colon C^2\cap L^2(\R^d)\times\R^d\to \R^{d\times d}$
is said to fulfill the \emph{Invariance Axioms},
if for any $\rho\in C^2\cap L^2(\R^d)$, $a\in\R^d$,
$\alpha\in\R\setminus\{0\}$, $A\in\gldr$, any parametrized function $\rho^{(t)}
= \sum_{k=1}^K \rho_k(\Cdot-a_k^{(t)})$, $t\ge 0,$ with $\rho_k\in
C^2\cap L^2(\R^d,\R)$, $a_k^{(t)}\in\R^d$, such that
$\|a_k^{(t)}-a_\ell^{(t)}\|\xrightarrow{t\to\infty}\infty$ for all $k\neq\ell$,
and any $y\in\R^d$,

\begin{enumerate}
  \item[(I1)]
  $\displaystyle
  \Phi(\rho(\Cdot\, - a),y+a) = \Phi(\rho,y)$
  \hfill
  (invariance under shifting),
  \item[(I2)]
  $\displaystyle
  \Phi(\alpha\cdot\rho,y) =   \alpha^{-1/(d+4)}\, \Phi(\rho,y)$
  \hfill
  (invariance under scalar multiplication),
  \item[(I3)]  
  $\displaystyle \phi_2 \phi_2^\intercal = A^{-1} \phi_1\phi_1^\intercal A^{-\intercal}$
  \hfill (invariance under scaling),\\[0.2cm]
  where $\phi_1 := \Phi(\rho,y),\ \phi_2 := \Phi\big(\abs{\det A}\, \rho(A\cdot\, \Cdot),A^{-1}y\big)$,
  \item[(I4)]
  $\Phi(\rho^{(t)},y+a_k^{(t)})  \xrightarrow{t\to\infty}  \Phi(\rho_k,y)$ for
  each $k=1,\dots,K$
  \hfill
  (locality).  
\end{enumerate}
\end{axiom}

\begin{figure}[H]
        \centering
        \begin{subfigure}[b]{0.8\textwidth}
                \centering
                \includegraphics[width=1\textwidth]{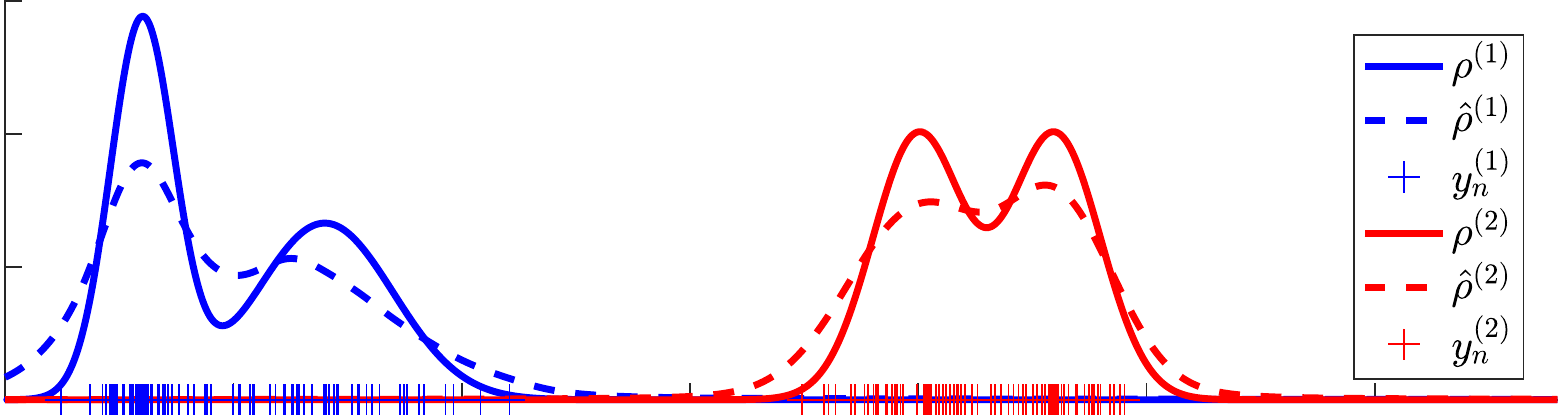}
                \end{subfigure}%
        \vspace{0.2cm}
        \begin{subfigure}[b]{0.8\textwidth}
                \centering
                \includegraphics[width=1\textwidth]{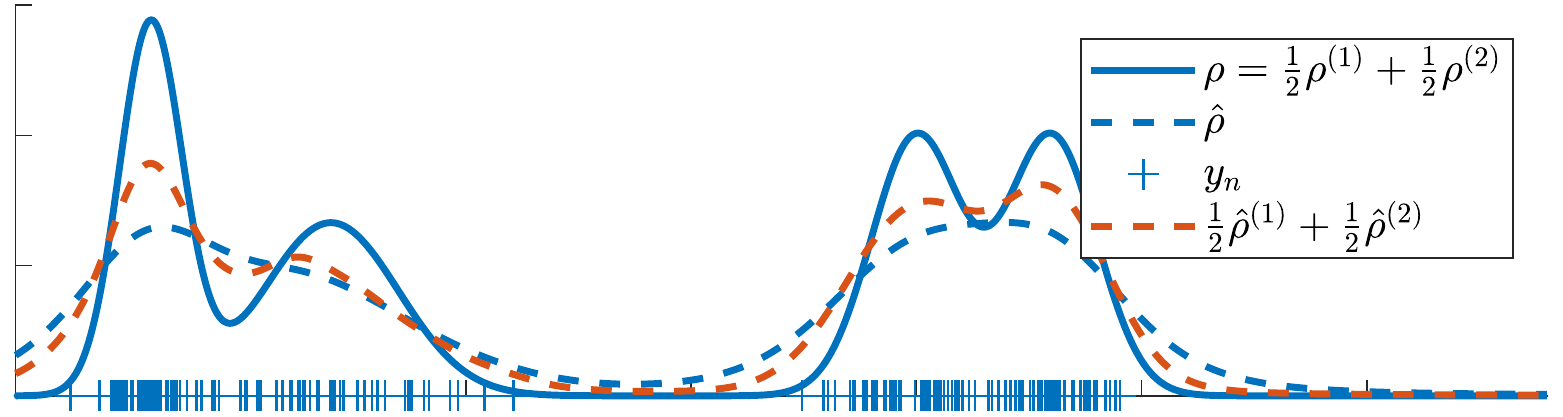}
        \end{subfigure}        
        \caption{The law \eqref{equ:hLaw} for $\beta = 1/d = 1$ applied to
        $\rho^{(1)}$ and $\rho^{(2)}$ separately (top) and $\rho = \rho^{(1)}/2 +
        \rho^{(2)}/2$ (bottom).
        We observe oversmoothing of the estimate $\hat\rho_\mth$ and inconsistency
        to the estimates $\hat \rho_{\mth^{(1)}}^{(1)},\, \hat
        \rho_{\mth^{(2)}}^{(2)}$. This effect can be amplified by making the
        density $\rho$ consist of more than two `well-separated parts'. Here,
        the proportionality constant in \eqref{equ:hLaw} was chosen manually.       
        }
        \label{fig:TwoDensitiesPlot}
\end{figure}

The reasoning for Invariance Axiom \ref{cond:ScalingConditions} (I2) has been
introduced in Section \ref{section:Introduction} and will now be discussed in detail.
Let $\rho^{(1)},\rho^{(2)}\in L^1(\R^d)$ be two probability densities with disjoint and `far-apart' supports $\Omega_1,\Omega_2\subset \R^d$ and $\rho^{(0)}(x) = \alpha\rho^{(1)}(x) + (1-\alpha)\rho^{(2)}(x)$, $\alpha\in[0,1]$.
Let $\cY^{(0)} = (y_1,\dots,y_{N_0})$ be independent $\rho^{(0)}$-distributed sample points and (after reordering) $\cY^{(1)} = \cY\cap\Omega_1=(y_1,\dots,y_{N_1})$ be those in $\Omega_1$. Then $\cY^{(1)}$ is $\rho^{(1)}$-distributed and (asymptotically) $N_1 \approx \alpha N_0$.
Assume we found suitable bandwidths $\mth^{(1)}$ for $\hat\rho_{\mth^{(1)}}^{(1)}$ based on $\cY^{(1)}$. Then, if we want to get an analogous result for  $\hat\rho_{\mth^{(0)}}^{(0)}$ in $\Omega_1$, $\hat\rho_{\mth^{(0)}}^{(0)}|_{\Omega_1} \approx \alpha\, \hat\rho_{\mth^{(1)}}^{(1)}$, and if we can neglect the influence of $\rho^{(2)}$ and $\cY^{(0)}\setminus \cY^{(1)}$ on $\hat\rho_{\mth^{(0)}}^{(0)}|_{\Omega_1}$, we have to require $h_n^{(0)}\approx h_n^{(1)}, \ n=1,\dots,N_1$, for the bandwidths $h_n^{(0)}$ of $\hat\rho_{\mth^{(0)}}^{(0)}$.
Since $N_1$ is different from $N_0$, we need to compensate for
\eqref{equ:mDependsOnN},
\[
N_1^{-\frac{1}{d+4}} \Phi(\rho^{(1)},y_n)
=
h_n^{(1)}
\stackrel{!}{\approx}
h_n^{(0)}
=
N_0^{-\frac{1}{d+4}} \Phi(\rho^{(0)},y_n)
=
{\underbrace{\left(\frac{N_1}{N_0}\right)}_{\approx\, \alpha}}^{\frac{1}{d+4}}
N_1^{-\frac{1}{d+4}} \Phi(\alpha\rho^{(1)},y_n),
\]
which is exactly what Invariance Axiom \ref{cond:ScalingConditions} (I2)
guarantees:
\[
\Phi(\alpha\rho^{(1)},y_n) = \alpha^{-1/(d+4)}\Phi(\rho^{(1)},y_n).
\]
This idea is visualized in Figure \ref{fig:TwoDensitiesPlot} and formulated more
rigorously in the following theorem, together with invariance of the
corresponding VKDE under shifting and scaling. 

Since our VKDE estimates are now based on different sampling families $\cY =
(y_1,\dots,y_N)$ and different bandwidths $\mth = (h_1,\dots,h_N)$, we
introduce the slightly more specific notation $\MM_K[\cY,\mth]$ instead of
$\hat\rho_\mth$,
\[
\hat\rho_\mth
=
{\rm MM}_K[\cY,\mth]
:=
\frac{1}{N}\sum_{n=1}^{N} \abs{\det h_n}^{-1}\, K\left(h_n^{-1}(\Cdot -y_n)\right),
\]
where $\MM_K$ stands for `mixture model with kernel K'.

\begin{theorem}
\label{theorem:InvariancePropertiesVKDE}
Let $\Phi\colon C^2\cap L^2(\R^d)\times\R^d\to \R^{d\times d}$ fulfill the Invariance Axioms \ref{cond:ScalingConditions} and
$K\in C^2(\R^d)$ be a radially symmetric probability density as in
\eqref{equ:radialKernel}.
\begin{enumerate}[(i)]
\item
Let $\rho_1\in C^2\cap L^2(\R^d)$ be a probability density, $a\in\R^d$, $A\in\gldr$
and
\[
\rho_2(x) = \rho_1(x-a),
\qquad
\rho_3(x) = \abs{\det A}\, \rho_1(Ax).
\]
Further, for $j=1,2,3$ let $\hat\rho_j = \MM_K\big[\cY_j,\mth_j\big]$, where
$\cY_j = (y_{j,n},\dots,y_{j,N})$ and $y_{1,n}\in\R^d,\ y_{2,n} = y_{1,n}+a,\
y_{3,n} = A^{-1}y_{1,n}$, $h_{j,n} = N^{-1/(d+4)}\, \Phi(\rho_j,y_{j,n})$, $n=1,\dots,N$. Then
\[
\hat\rho_2(x) = \hat\rho_1(x-a),
\qquad
\hat\rho_3(x) = \abs{\det A}\, \hat\rho_1(Ax).
\]
\item
For $k=1,\dots,K$, let $\rho_k\in C^2\cap L^2(\R^d,\R)$ be densities
and $\cY_k = (y_{k,1},\dots,y_{k,N_k})$.
Further, let $N_0 = \sum_{k=1}^K N_k$, $\alpha_k = N_k/N_0$, $\cY_0 :=
\bigcup_{k=1}^K (a_k^{(t)}+Y_k)$, let $a_k^{(t)}\in\R^d$, $t\ge 0$, such that
$\|a_k^{(t)}-a_\ell^{(t)}\|\xrightarrow{t\to\infty}\infty$ for all $k\neq
\ell$, and
\[
\rho_0 = \sum_{k=1}^K \alpha_k\rho_k(\Cdot-a_k^{(t)}).
\]
%
Finally, for $k=0,\dots,K$, let $h_{k,n} = N_k^{-1/(d+4)}\Phi(\rho_k,y_{k,n}),\
n=1,\dots,N_k$ and $\hat\rho_k = \MM_K[\cY_k,\mth_k]$.
Then, asymptotically for $t\to\infty$, 
\[
\hat \rho_0 = \sum_{k=1}^K \alpha_k\hat\rho_k(\Cdot-a_k^{(t)}).
\]
More precisely, for each $x\in\R^d$ and $k=1,\dots,K$,
$
\hat \rho_0(x+a_k^{(t)})
\xrightarrow{t\to\infty}
\alpha_k \hat\rho_k(x).
$
\end{enumerate}
\end{theorem}

\begin{proof}
\begin{enumerate}[(i)]
\item
The proof of $\hat\rho_2(x) = \hat\rho_1(x-a)$ is straightforward.
Invariance Axiom \ref{cond:ScalingConditions} (I3) implies
$h_{3,n}^{\vphantom{\intercal}} h_{3,n}^\intercal = A^{-1}
h_{1,n}^{\vphantom{\intercal}} h^\intercal_{1,n} A^{-\intercal}$ and together
with equation \eqref{equ:radialKernel} this yields
\begin{align*}
\hat \rho_3(x)
&=
\frac{1}{N}\sum_{n=1}^{N}\abs{\det h_{3,n}^{-1}}\, 
\gamma\Big(\norm{h_{3,n}^{-1} (x-A^{-1}y_{1,n})}_2^2\Big)
\\
&=
\frac{\abs{\det A}}{N}\sum_{n=1}^{N}\abs{\det h_{1,n}^{-1}} 
\, \gamma\Big((x-A^{-1}y_{1,n})^\intercal A^\intercal h_{1,n}^{-\intercal}
h_{1,n}^{-1} A \, (x-A^{-1}y_{1,n})\Big)
\\
&=
\frac{\abs{\det A}}{N}\sum_{n=1}^{N}\abs{\det h_{1,n}^{-1}}
\, \gamma\Big(\norm{h_{1,n}^{-1} (Ax-y_{1,n})}_2^2\Big)
\\
&=\abs{\det A}\, \hat \rho_1(Ax).
\end{align*}
\item
Let $\nu=1,\dots,N_0$ and $y_{0,\nu}\in\cY_0$, i.e. $y_{0,\nu} = a_k^{(t)} +
y_{k,n}$ for some $k=1,\dots,K$, $n=1,\dots,N_k$. Then Invariance Axioms
\ref{cond:ScalingConditions} (I2) and (I4) imply
\[
h_{0,\nu}
=
N_0^{-1/(d+4)}\Phi(\rho_0,y_{0,\nu})
\xrightarrow{t\to\infty}
N_0^{-1/(d+4)}\Phi(\alpha_k\rho_k,y_{k,n})
=
N_k^{-1/(d+4)}\Phi(\rho_k,y_{k,n})
=
h_{k,n}.
\]
Since $K(x)\to 0$ for $\norm{x}\to\infty$, we obtain for $k=1,\dots,K$:
\begin{align*}
\hat \rho_0(x+a_k^{(t)})
=\ 
&\frac{1}{N_0}\sum_{\nu=1}^{N_0} \abs{\det h_{0,\nu}}^{-1}\,
K\left(h_{0,\nu}^{-1}(x+a_k^{(t)}-y_{0,\nu})\right)
\\
\xrightarrow{t\to\infty}\ 
&\frac{\alpha_k}{N_k} \sum_{n=1}^{N_k} \abs{\det
h_{k,n}}^{-1}\, K\left(h_{k,n}^{-1}(x-y_{k,n})\right)
\hspace{0.2cm}
=
\alpha_k \hat\rho_k(x).
\end{align*}
\end{enumerate}
\end{proof}

\begin{remark}
\label{rem:InvariancePropertiesVKDE}
Theorem \ref{theorem:InvariancePropertiesVKDE} can be summarized as follows:
\\
The Invariance Axioms \ref{cond:ScalingConditions} guarantee invariance of the
VKDE under shifting and scaling of the original density (and,
correspondingly, the sample points) and, if a density $\rho_0$ is a convex
combination of densities $\rho_k,\ k=1,\dots,K$, with `far apart' supports, its VKDE is approximately the convex combination of the single VKDE's (based on the sampling points lying in the corresponding
supports),
\[
\rho_0
=
\sum_{k=1}^{K}
\alpha_k\, \rho_k
\qquad
\text{implies}
\qquad
\hat\rho_0
\approx
\sum_{k=1}^{K}
\alpha_k\, \hat\rho_k.
\]
\end{remark}

\subsection{Choice of the Bandwidths}
\label{section:BandwidthsChoice}
In view of Proposition \ref{prop:muFulfillsAdaptationConditions} it is tempting
to set $h_n \propto N^{-1/(d+4)}\mu_\rho^{-1}(y_n)$, however, in order to
account for the essential difference between adaptive convolutions and VKDE
mentioned above and characterized by Invariance Axioms
\ref{cond:ScalingConditions} (I2), this choice has to be adjusted in the
following way:
\begin{theorem}
Let $\kappa>0$. The following choice for the bandwidths fulfills the Invariance
Axioms \ref{cond:ScalingConditions}:
\begin{equation}
\label{equ:FinalChoiceBandwidths}
\framebox[11.4cm][c]{
$\displaystyle 
h_n = N^{-1/(d+4)} \Phi(\rho,y_n),
\qquad
\Phi(\rho,y) = \abs{\frac{\kappa\, \rho(y)}{\det \mu_\rho (y)}}^{-1/(d+4)}
\mu_\rho^{-1} (y),
$
}
\end{equation}
\end{theorem}
\begin{proof}
The proof follows directly from Proposition
\ref{prop:muFulfillsAdaptationConditions} (we adopt the notation from the
Invariance Axioms \ref{cond:ScalingConditions}):

\begin{align*}
\Phi(\rho(\Cdot\, - a),y+a)
=\ &
\abs{\frac{\kappa\, \rho(y_n)}{\det \mu_{\rho(\Cdot\, - a)} (y_n+a)}}^{-\frac{1}{d+4}}
\mu_{\rho(\Cdot\, - a)}^{-1} (y_n+a)
=
\abs{\frac{\kappa\, \rho(y_n)}{\det \mu_\rho (y_n)}}^{-\frac{1}{d+4}}
\mu_\rho^{-1} (y_n)
\\ =\ &
\Phi(\rho,y),
\\[0.2cm]
\Phi(\alpha\cdot\rho,y)
=\ &
\abs{\frac{\kappa\, \alpha\, \rho(y_n)}{\det \mu_{\alpha\cdot\rho} (y_n)}}^{-\frac{1}{d+4}}
\mu_{\alpha\cdot\rho}^{-1} (y_n) 
=
\abs{\frac{\kappa\, \alpha\, \rho(y_n)}{\det \mu_{\rho} (y_n)}}^{-\frac{1}{d+4}}
\mu_{\alpha\cdot\rho}^{-1} (y_n)
=
\alpha^{-1/(d+4)}\, \Phi(\rho,y),
\end{align*}
\begin{align*}
\phi_2 \phi_2^\intercal
=\ &
\abs{\frac{\kappa\, \abs{\det A}\, \rho(Ay_n')}{\det \mu_{\abs{\det A}\, \rho(A\cdot\, \Cdot)} (y_n')}}^{\frac{2}{d+4}}
\left(\mu_{\abs{\det A}\, \rho(A\cdot\, \Cdot)}^\intercal (y_n')\, 
\mu_{\abs{\det A}\, \rho(A\cdot\, \Cdot)} (y_n')\right)^{-1}
\\
=\ &
\abs{\frac{\kappa\, \abs{\det A}\, \rho(y_n)}{\det A\, \det \mu_{\rho} (y_n)}}^{\frac{2}{d+4}}
\left(
A^\intercal \mu_{\rho}^\intercal (y_n)
\mu_{\rho} (y_n)\, A
\right)^{-1}
=
A^{-1}\phi_1\phi_1^\intercal A^{-\intercal},
\\[0.3cm]
\Phi(\rho^{(t)},y+a_k^{(t)})
=\ &
\abs{\frac{\kappa\, \rho^{(t)}(y+a_k^{(t)})}{\det \mu_{\rho^{(t)}}
(y+a_k^{(t)})}}^{-1/(d+4)} \mu_{\rho^{(t)}}^{-1} (y+a_k^{(t)}),
\xrightarrow{t\to\infty}
\abs{\frac{\kappa\, \rho_k(y)}{\det \mu_{\rho_k} (y)}}^{-1/(d+4)}
\mu_{\rho_k}^{-1} (y)
\\
=\ &
\Phi(\rho_k,y).
\end{align*}

\end{proof}

It is remarkable that Parzen's equation \eqref{equ:ParzenBandwidth} from 1962
also fulfills all the Invariance Axioms \ref{cond:ScalingConditions} in the
univariate case.
However, its generalization to the multivariate case encounters some fundamental
problems, which we discuss in the following section.

%% file: sections/Parzen.tex
\section{Parzen's Law \eqref{equ:ParzenBandwidth} in the Multivariate Case}
\label{section:Parzen}

A standard criterion for optimizing the bandwidth $h$ is to minimize the MISE.
However, in the VKDE setting, minimizing the MSE and MISE is asymptotically
equivalent: As the number $N$ of samples grows (and the bandwidths
decrease), the contribution of far away kernels can be neglected and optimizing the
bandwidths locally by minimizing the MSE automatically results in minimizing the
MISE as well (asymptotically).
\begin{quote}
``Recall that the MISE accumulates pointwise errors. Thus accumulating the minimal
pointwise errors [...] gives the asymptotic lower bound to the
adaptive AMISE.'' \cite{scott2015multivariate}
\end{quote}
This observation strongly simplifies the choice of the asymptotically optimal
bandwidths, which now can be optimized separately as in Parzen's law \eqref{equ:ParzenBandwidth}, instead of simultaneously. Let us try to generalize this law to arbitrary dimension $d$.
In this section, we will follow the discussion in \cite[Section
6.6.3]{scott2015multivariate} and \cite[Section 5]{terrell1992variable}, where
several of the results presented here have already been derived.

In this section, we assume that $\rho\in C^4(\R^d)$. As always, the kernel
$K\in C^2(\R^d)$ is assumed to be to be a radially symmetric probability density function as in \eqref{equ:radialKernel} and in addition we assume the following conditions on the second, third and fourth moments of $K$ (here, $i,j,k,l\in\{1,\dots,d\}$ and $\delta$ denotes the Kronecker delta function),
\begin{equation}
\label{equ:MomentConditions}
\small
\int x_i x_j K(x)\, \mathrm dx
=
\delta_{ij},
\quad
\int x_i x_j x_k K(x)\, \mathrm dx
=
0,
\quad
\int x_i x_j x_k x_l K(x)\, \mathrm dx
=
\delta_{ij}\delta_{kl} + \delta_{ik}\delta_{jl} + \delta_{il}\delta_{jk},
\end{equation}
which are fulfilled by e.g. the standard Gaussian density function. We also assume that
\begin{equation}
\label{equ:formulaRK}
R(K):=\int K(t)^2\, \mathrm dt < \infty,
\end{equation}
e.g. for a standard Gaussian we have $R(K) = (4\pi)^{-d/2}$, and denote for
$h\in\gldr$
\begin{equation}
\Delta_4(\rho,h)
:=
\sum_{i,j,k,l=1}^{d}
\frac{\partial^4\rho}{\partial_{x_i}\partial_{x_j}\partial_{x_k}\partial_{x_l}}
\left(
\big[h^2\big]_{ij}\big[h^2\big]_{kl}
+\big[h^2\big]_{ik}\big[h^2\big]_{jl}
+\big[h^2\big]_{il}\big[h^2\big]_{jk}
\right),
\end{equation}
where $[A]_{ij}$ denotes the entry in the $i$-th row and $j$-th column of a matrix $A$.
\begin{proposition}
Let the kernel $K\in C^2(\R^d)$ fulfill \eqref{equ:radialKernel},
\eqref{equ:MomentConditions} and \eqref{equ:formulaRK}, $\rho\in C^4(\R^d)$ be a
probability density function and $H_x = D^2\rho(x)$ denote the Hessian of $\rho$
at $x\in\R^d$. Then, asymptotically for large $N$ and small $\|h\|$, the bias,
variance and MSE of
\[
\hat{\rho} (x) 
:=
\frac{1}{N}\sum_{n=1}^{N} K_h\left(x-y_n\right)
:=
\frac{|\det h^{-1}|}{N}\sum_{n=1}^{N} K\left(h^{-1}(x-y_n)\right)
\]
in dependence of $h\in\gldr$ are given by
\begin{align*}
\B_h[\hat \rho(x)]
&=
\tfrac{1}{2}\tr(h^\intercal H_x h)
+
\tfrac{1}{24}\Delta_4(\rho,h)(x),
\\[0.2cm]
\V_h[\hat \rho(x)]
&=
\frac{\rho(x)\, R(K)}{N\, \abs{\det h}},
\\
\MSE_h[\hat\rho(x)]
&=
\frac{\rho(x)\, R(K)}{N\, \abs{\det h}}
+
\left(
\tfrac{1}{2}\tr\left({h^\intercal H_x h}\right)
+
\tfrac{1}{24}\Delta_4(\rho,h)(x)
\right)^2.
\end{align*}
\end{proposition}

\begin{remark}
We expanded the bias up to order $4$, since $\tr(h^\intercal H_x h)$ will turn out to be zero in certain cases. In all other cases the term $\tfrac{1}{24}\Delta_4(\rho,h)(x)$ will be neglected.
\end{remark}

\begin{proof}
Since the sample points $y_1,\dots,y_N\stackrel{\rm iid}{\sim}\rho$ are independent, a Taylor expansion yields (we abbreviate ``higher order terms'' by $\hot$)
\begin{align*}
\B_h[\hat \rho(x)]
&=
\abs{\det h^{-1}}\int K(h^{-1}(x-y))\, \rho(y)\, \mathrm dy - \rho(x)
\\
&=
\int K(t)\,  \left[\rho(x-h t) - \rho(x)\right]\, \mathrm dt
\\
&=
\int K(t)\, \Big[
\sum_{\ell=1}^{4}
\frac{(-1)^\ell}{\ell !}
\left(D^{(\ell)}\rho(x)\right)(\underbrace{ht,\dots,ht}_{\ell\text{ times}})
+ \hot
\Big]\, \mathrm dt
\\
&=
\tfrac{1}{2}\tr(h^\intercal H_x h)
+
\tfrac{1}{24}\Delta_4(\rho,h)(x)
\ + \hot\, ,
\\[0.3cm]
\V_h[\hat \rho(x)]
&=
\frac{1}{N}\left[
\abs{\det h^{-1}}^2\int K(h^{-1}(x-y))^2\, \rho(y)\, \mathrm dy
- (\rho(x) + \B_h[\hat \rho(x)])^2
\right]
\\
&=
\frac{\abs{\det h^{-1}}}{N}
\int K(t)^2\, \rho(x-ht)\, \mathrm dt
- \frac{1}{N}\left(\rho(x) + o(\|h\|^{2})\right)^2
\\
&=
\frac{\rho(x)\, R(K)}{N\, \abs{\det h}}\ +\ \hot\, ,
\end{align*}
which also proves the formula for $\MSE_h[\hat\rho(x)]$.
\end{proof}

In order to minimize $\MSE_h[\hat\rho(x)]$ we will discuss three scenarios
concerning the eigenvalues of the Hessian $H_x = D^2\rho(x)$. This is a fundamental difference compared
to the univariate case, in which (basically) only the first case has to be
covered.

\begin{enumerate}[\textbf{{Case} 1:}]
\item $H_x$ is either positive definite or negative definite.
Then, ignoring the higher order terms as well as the term $\tfrac{1}{24}\Delta_4(\rho,h)(x)$ and taking the derivative of $\MSE_h[\hat\rho(x)]$ with respect to $h$, we obtain the following condition for its minimizer $h_\opt$:
\[
\frac{\rho(x)\, R(K)}{N}\, \Id = \det h_\opt\, \tr(h_\opt^\intercal H_x h_\opt)\, h_\opt^\intercal H_x h_\opt,
\]
which is solved by
\[
h_\opt = \left(\frac{\rho(x)\, R(K)\, \abs{\det H_x}^{1/2}}{d\, N}\right)^{1/(d+4)} (\pm\, H_x)^{-1/2}.
\]
\item[\textbf{{Case} 2:}] $H_x$ has both positive and negative eigenvalues.
We then rewrite $h = \lambda B$, $\lambda>0$ and $B\in\R^{d\times d}$ with $\det B =1$. Now, $B$ does not influence the variance $\V_h[\hat \rho(x)]$ and can be chosen to eliminate the leading term of the bias: $\tr(h^\intercal H_x h) = 0$. This can be realized by the following (non-unique) choice.
\begin{compactitem}
\item Diagonalize $H_x$ by an orthogonal matrix $U\in O(d)$:
\begin{equation*}
H_x = 
U
\, \diag (\delta_1,\dots,\delta_k,-\delta_{k+1},\dots,-\delta_d)\, 
U^\intercal,
\qquad
\delta_j\ge 0,\quad j=1,\dots,d
\end{equation*}
\item
Choose $B = (\det \tilde B)^{-1/d} \tilde B$, where
\begin{equation*}
\tilde B
= 
U
\, \diag (\delta_1^{-1/2},\dots,\delta_k^{-1/2},\delta^{-1/2},\dots,\delta^{-1/2})\, 
U^\intercal,
\qquad
\delta := \frac{\sum_{j=k+1}^{d}\delta_j}{k}
\end{equation*}
\item Then $\det B = 1$ by construction and
\begin{equation*}
(\det \tilde B)^{2/d}\, \tr (B^\intercal H_x B)
=
\tr (\tilde B^\intercal H_x \tilde B)
=
\sum_{j=1}^{k} 1\ -\ \sum_{j=k+1}^{d}\delta^{-1}\delta_j
=
k-k
=
0.
\end{equation*}
\end{compactitem}
$B$ being chosen, $\lambda >0$ can now be chosen to (asymptotically) minimize
\begin{equation*}
\MSE_\lambda[\hat\rho(x)]
\approx
\frac{\rho(x)\, R(K)}{N\, \lambda^d}
+
\left(
\frac{\lambda ^4}{24}\Delta_4(\rho,B)(x)
\right)^2,
\end{equation*}
resulting in (provided $\Delta_4(\rho,B)(x)\neq 0$)
\begin{equation}
\label{equ:optimalLambda}
\lambda_\opt = \left(\frac{24^2\, d\, \rho(x)\, R(K)}{N\, \Delta_4(\rho,B)(x)^2}\right)^{1/(d+8)}.
\end{equation}
\item[\textbf{{Case} 3:}] $H_x$ is positive or negative semidefinite with some eigenvalues equal to zero. The minimization problem can then be reduced to a lower dimension, see \cite[Section 6.6.3]{scott2015multivariate}. This case is degenerate and therefore of little practical relevance.
\end{enumerate}

\begin{remark}
\label{remark:arguableChoiceB}
The choice of $B$ in Case 2 is not the only way to achieve $\tr(h^\intercal H_x h) = 0$ and is based entirely on intuitive reasoning on how to rotate the kernel in space. However, as long as $\tr(h^\intercal H_x h) = 0$ is guaranteed, this choice is irrelevant for the asymptotic analysis since
``regions where the density is saddle-shaped asymptotically contribute nothing to the AAMISE [asymptotic adaptive MISE] compared to regions where the density is definite'' \cite{scott2015multivariate}.
\end{remark}

We will refer to this construction as Parzen's VKDE.
While this approach is provably optimal \emph{asymptotically}, it shows serious
problems for finite sample sizes, as discussed in the next section.

%% file: sections/NumericalExperiments.tex
\section{Numerical Experiments}
\label{section:NumericalExperiments}

In order to compare the different kernel density estimates, we will present an
artificial example of a highly curved density (a strongly deformed Gaussian),
which is particularly hard to approximate, as well as a real data example in
Section \ref{section:Earthquake} (both in two dimensions).

\subsection{Artificial Example}
\label{section:ArtificialExample}
Consider the following ``complicated'' density $\rho$, visualized in Figure
\ref{fig:ComparisonVKDEs}(a),
\begin{equation}
\label{equ:bananaDensity}
\rho(x)
=
\frac{1}{2\pi \sigma}\, 
\exp\left(-\frac{1}{2} \left[\Big(\frac{x_1}{\sigma}\Big)^2 + \bigg(x_2 - \alpha\Big(\frac{x_1}{\sigma}\Big)^2\bigg)^2\right]  \right),
\qquad
\alpha = 4,\ \sigma = 5.
\end{equation}
We demonstrate the performance of several kernel density estimates when
making use of $\rho$ and its derivatives (Figure \ref{fig:ComparisonVKDEs}
(c)--(e)) and without this knowledge (Figure \ref{fig:ComparisonVKDEs} (b),(f)
and Figure \ref{fig:fpi2d}), see the discussion in Section \ref{section:FixedPointIteration} and the technical details in Appendix
\ref{section:technicalDetails}.

\begin{figure} 
        \centering
        \begin{subfigure}[b]{0.26\textwidth}
            \centering
			\includegraphics[width=\textwidth]{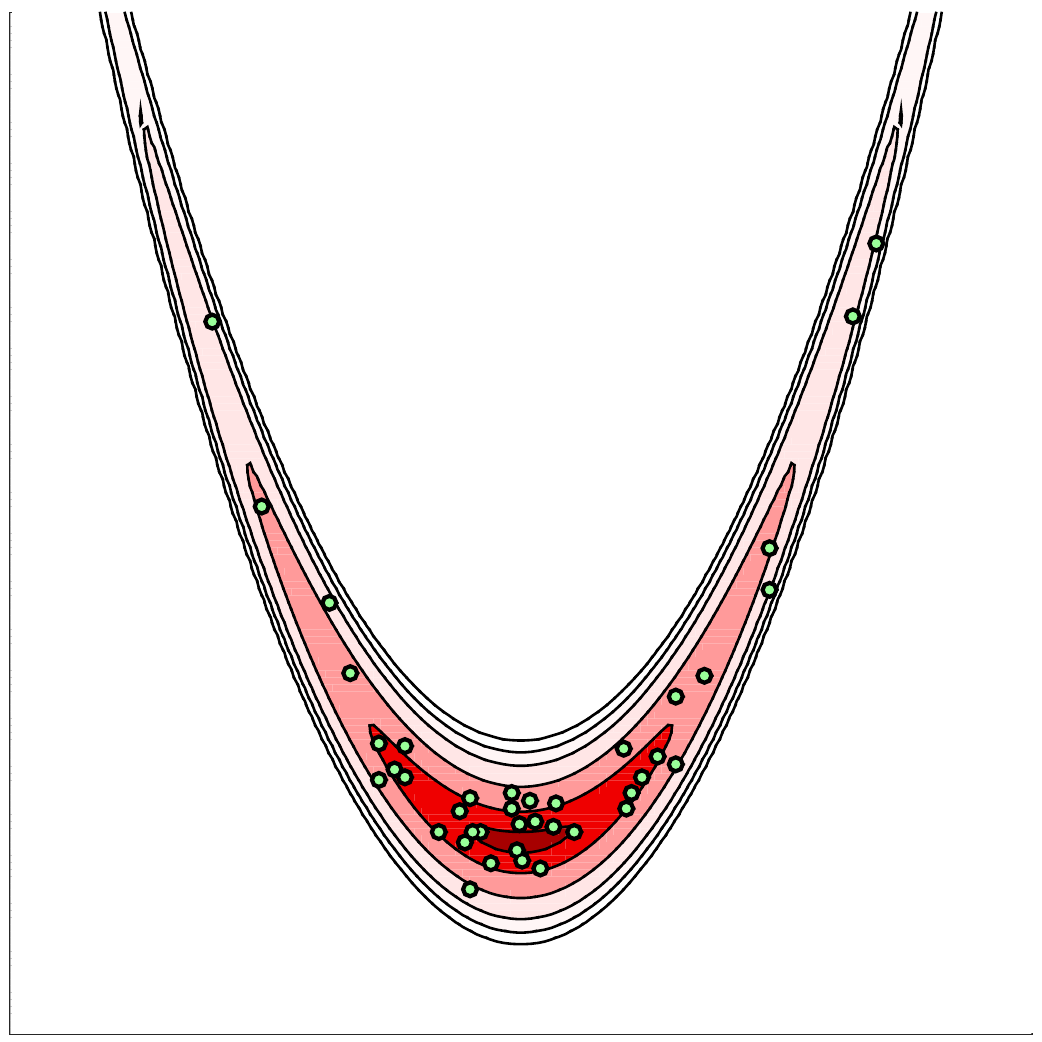}
            \caption{True density}
        \end{subfigure}
        \hspace{0.5cm}
        \begin{subfigure}[b]{0.26\textwidth}
            \centering
			\includegraphics[width=\textwidth]{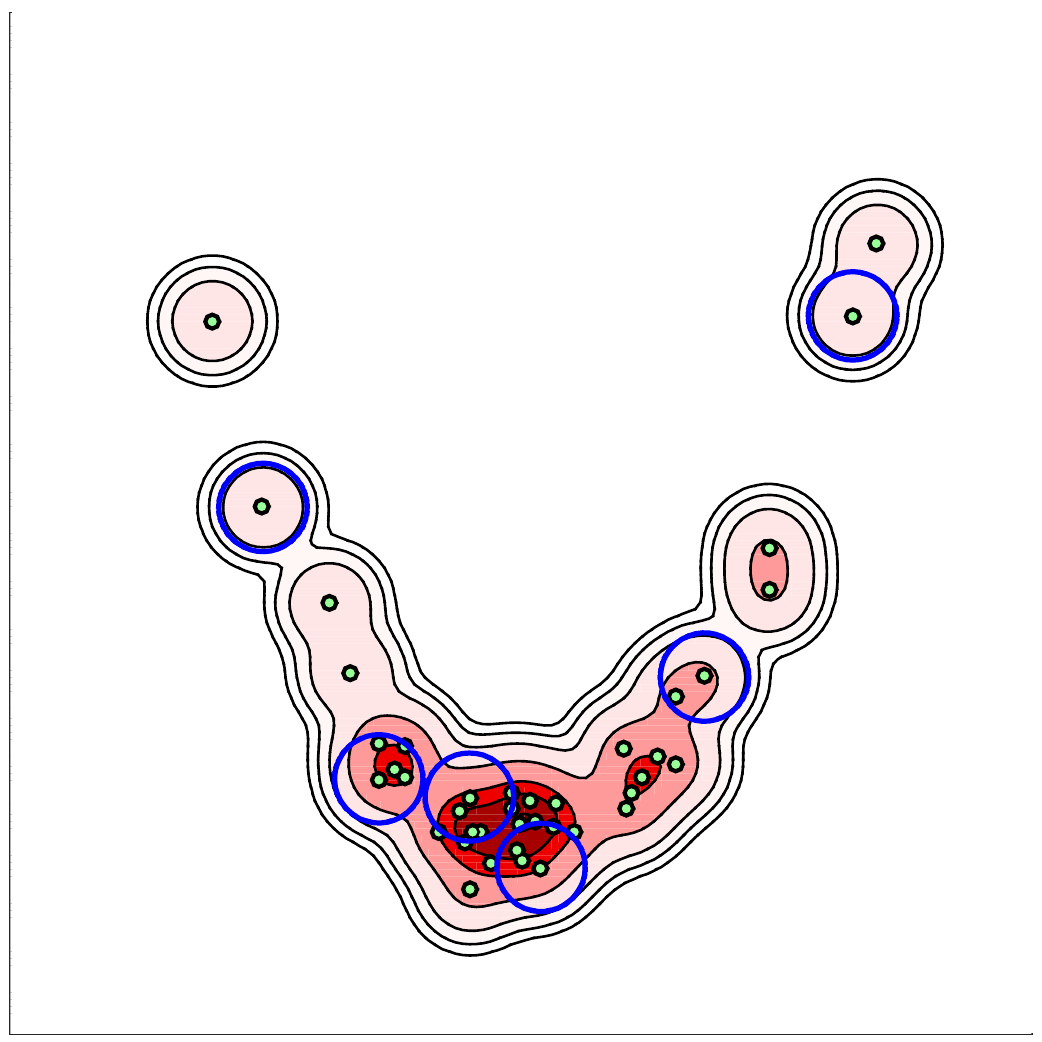}
			\caption{Standard KDE}
		\end{subfigure}
        \hspace{0.5cm}
        \begin{subfigure}[b]{0.26\textwidth}
            \centering
			\includegraphics[width=\textwidth]{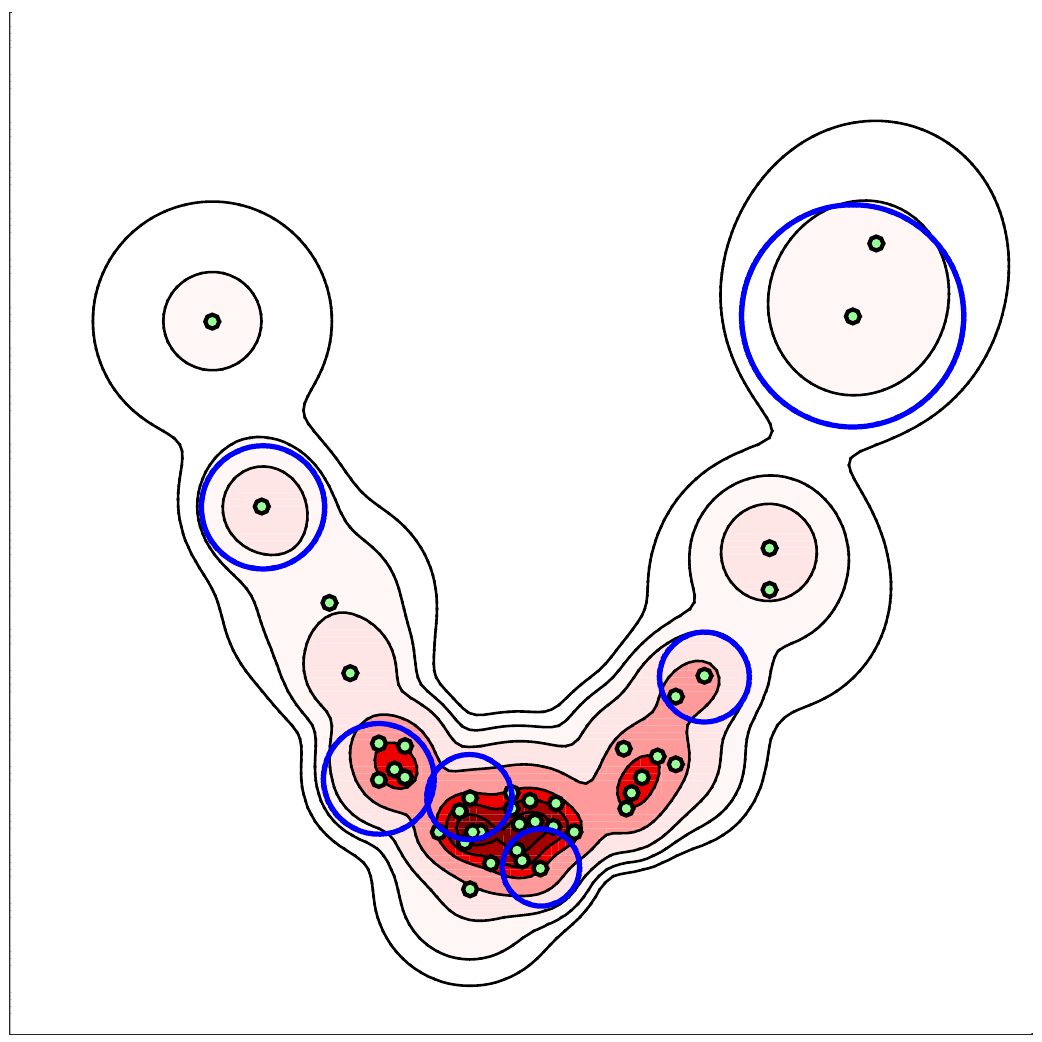}
			\caption{VKDE using \eqref{equ:hLaw}}
        \end{subfigure}
        \vfill
        \begin{subfigure}[b]{0.26\textwidth}
            \centering
			\includegraphics[width=\textwidth]{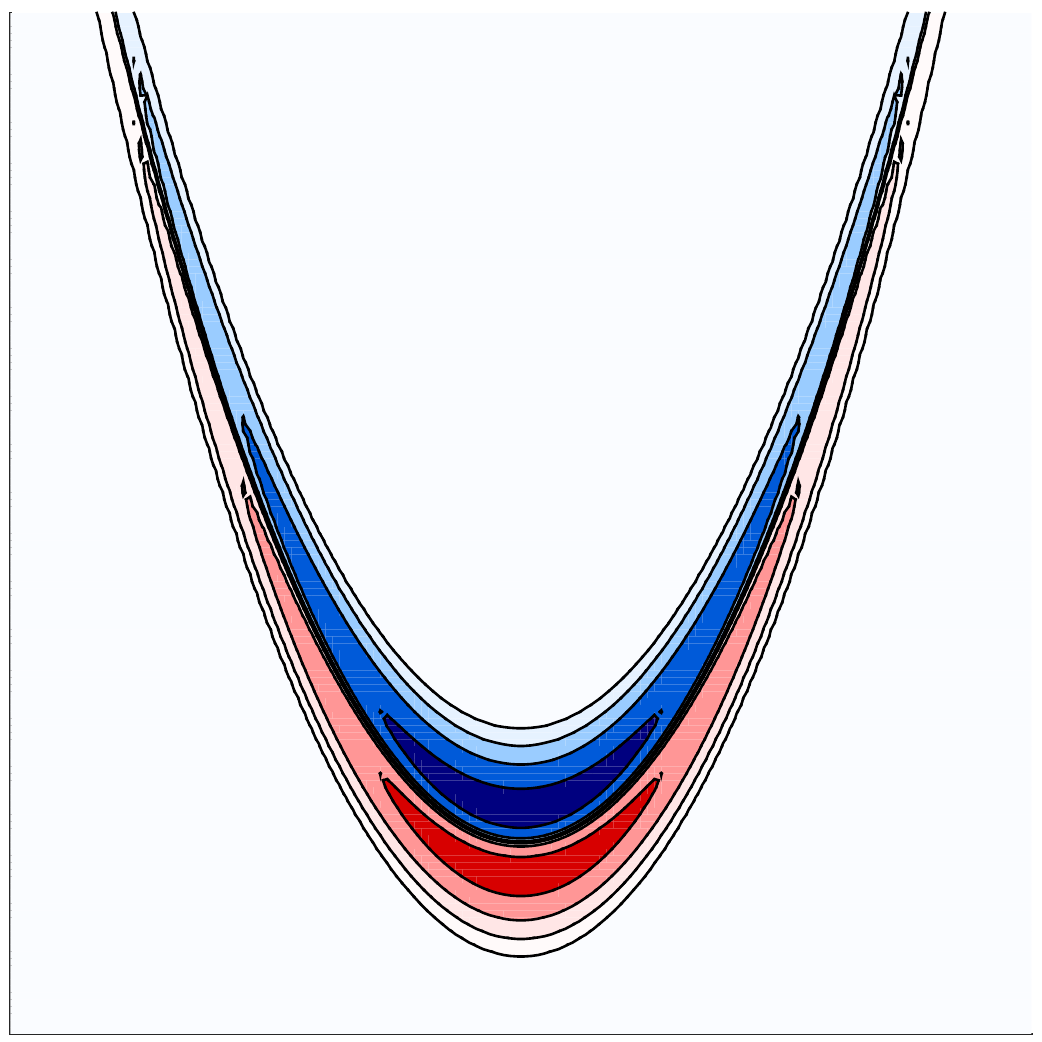}
        \end{subfigure}
        \hspace{0.5cm}
        \begin{subfigure}[b]{0.26\textwidth}
            \centering
			\includegraphics[width=\textwidth]{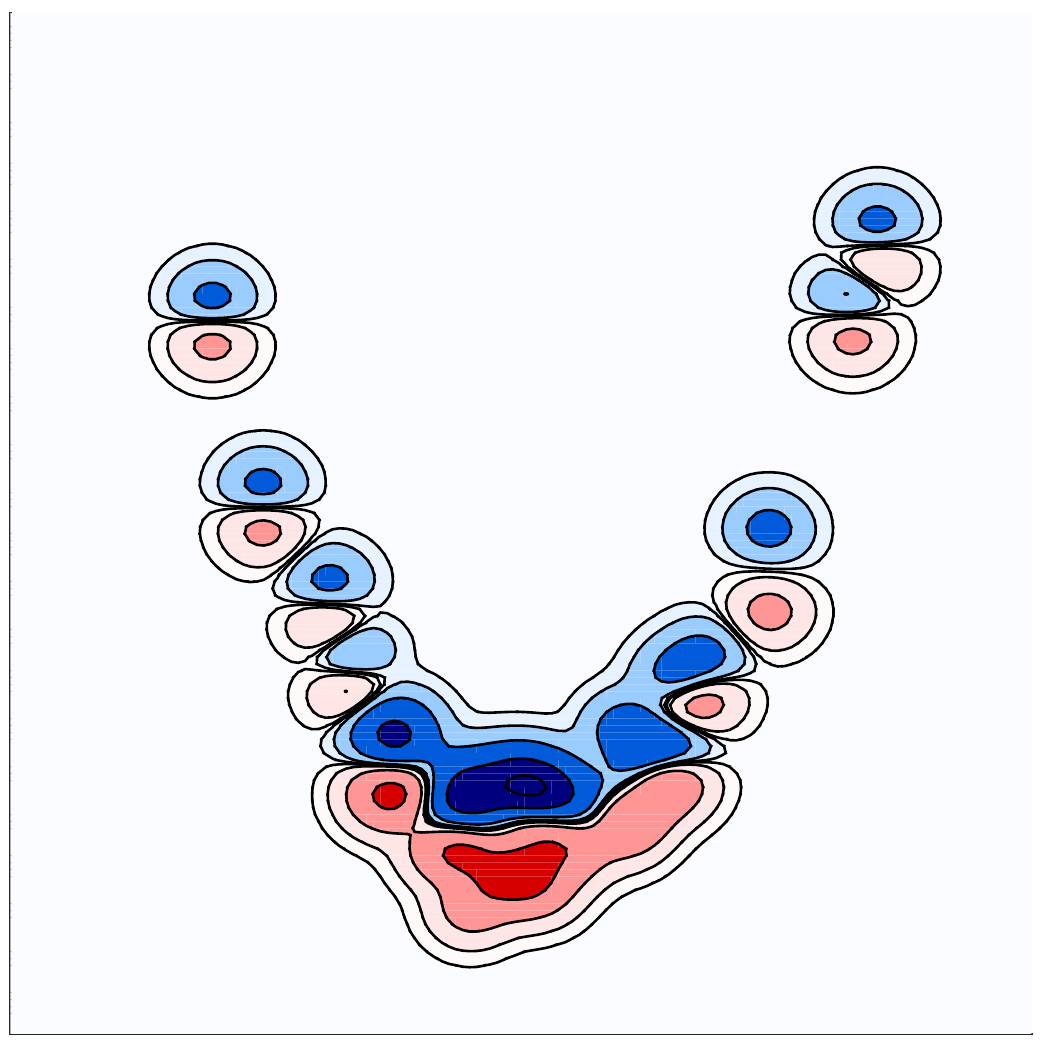}
		\end{subfigure}
        \hspace{0.5cm}
        \begin{subfigure}[b]{0.26\textwidth}
            \centering
			\includegraphics[width=\textwidth]{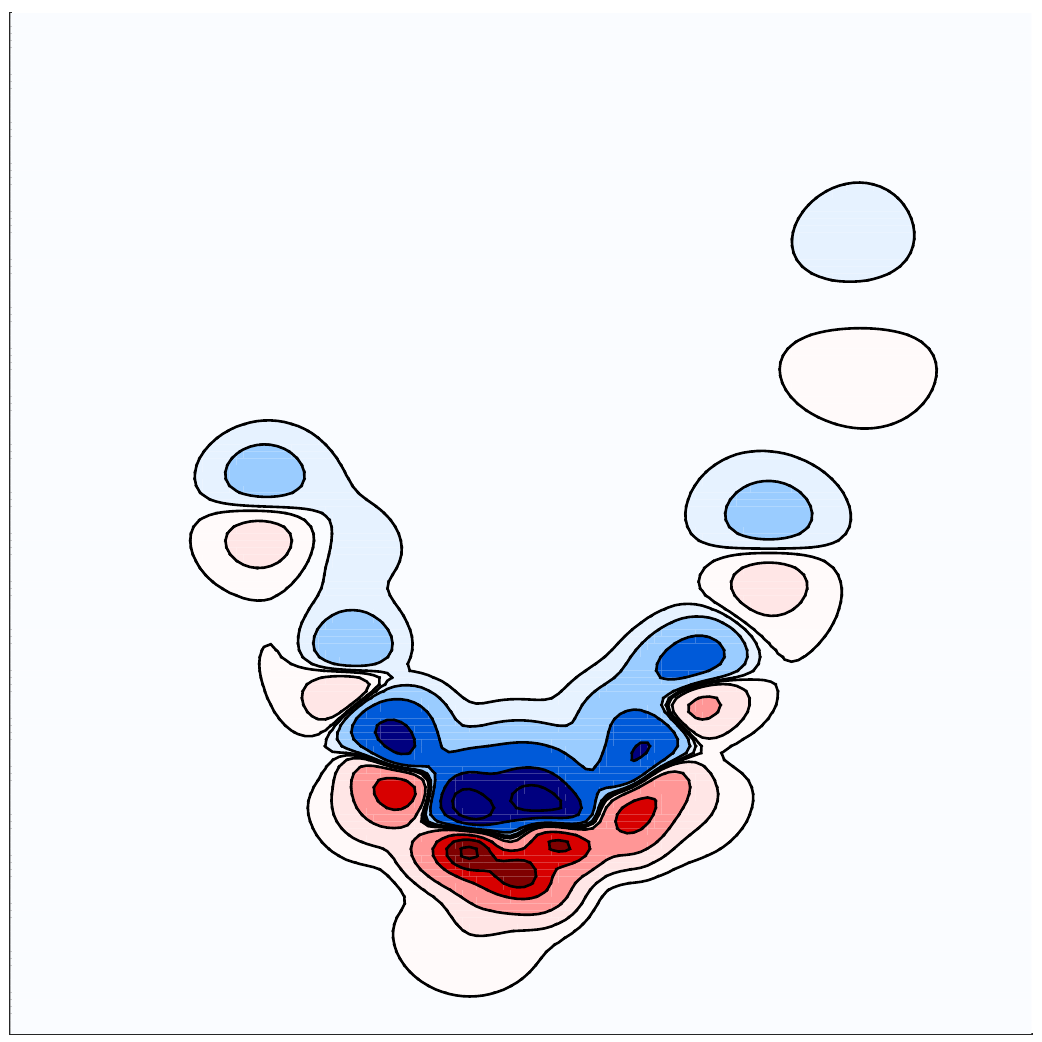}
        \end{subfigure}
        \vfill
		\vspace{0.2cm}               
        \hrule
        \vspace{0.04cm}               
        \hrule
        \vspace{0.2cm}               
        \begin{subfigure}[b]{0.26\textwidth}
            \centering
			\includegraphics[width=\textwidth]{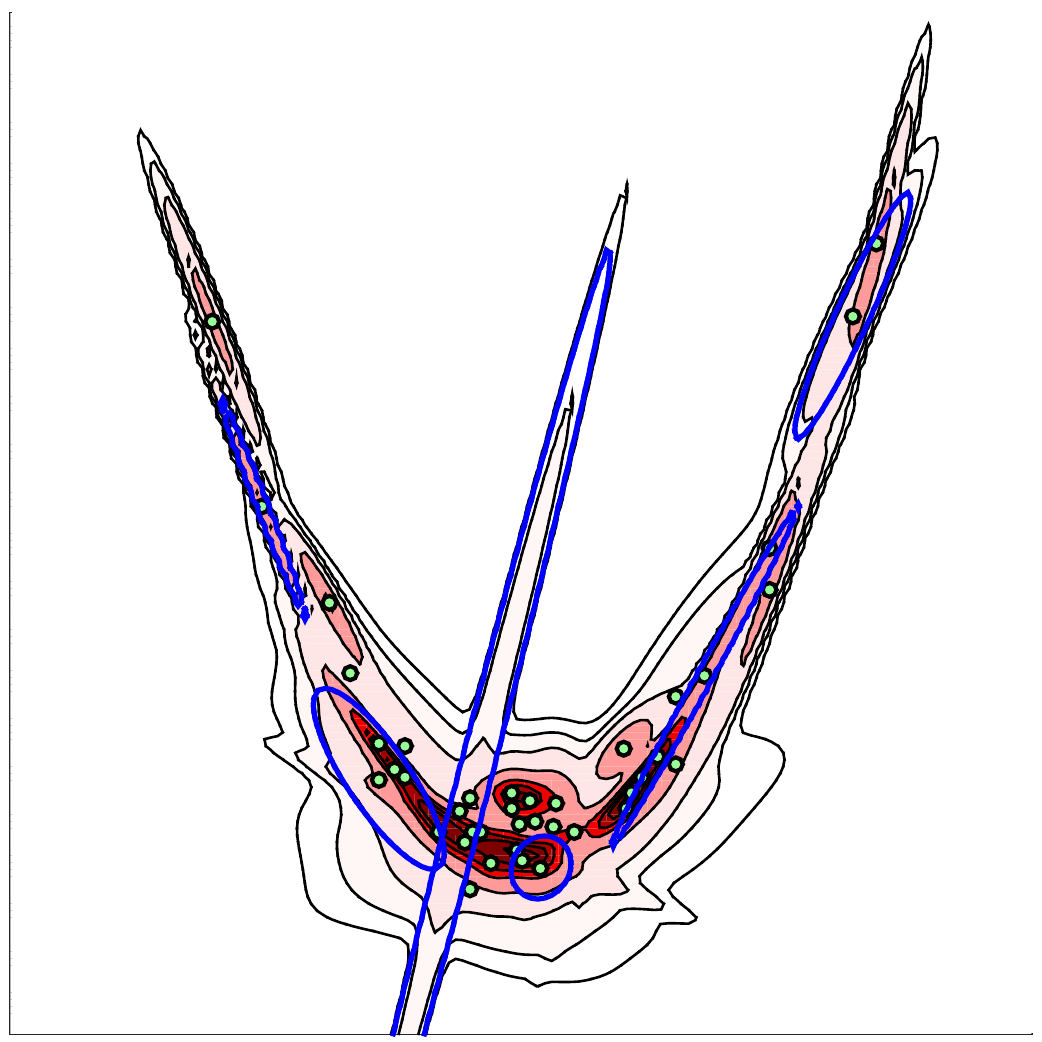}
			\caption{Parzen's VKDE}
		\end{subfigure}
        \hspace{0.5cm}
        \begin{subfigure}[b]{0.26\textwidth}
            \centering
			\includegraphics[width=\textwidth]{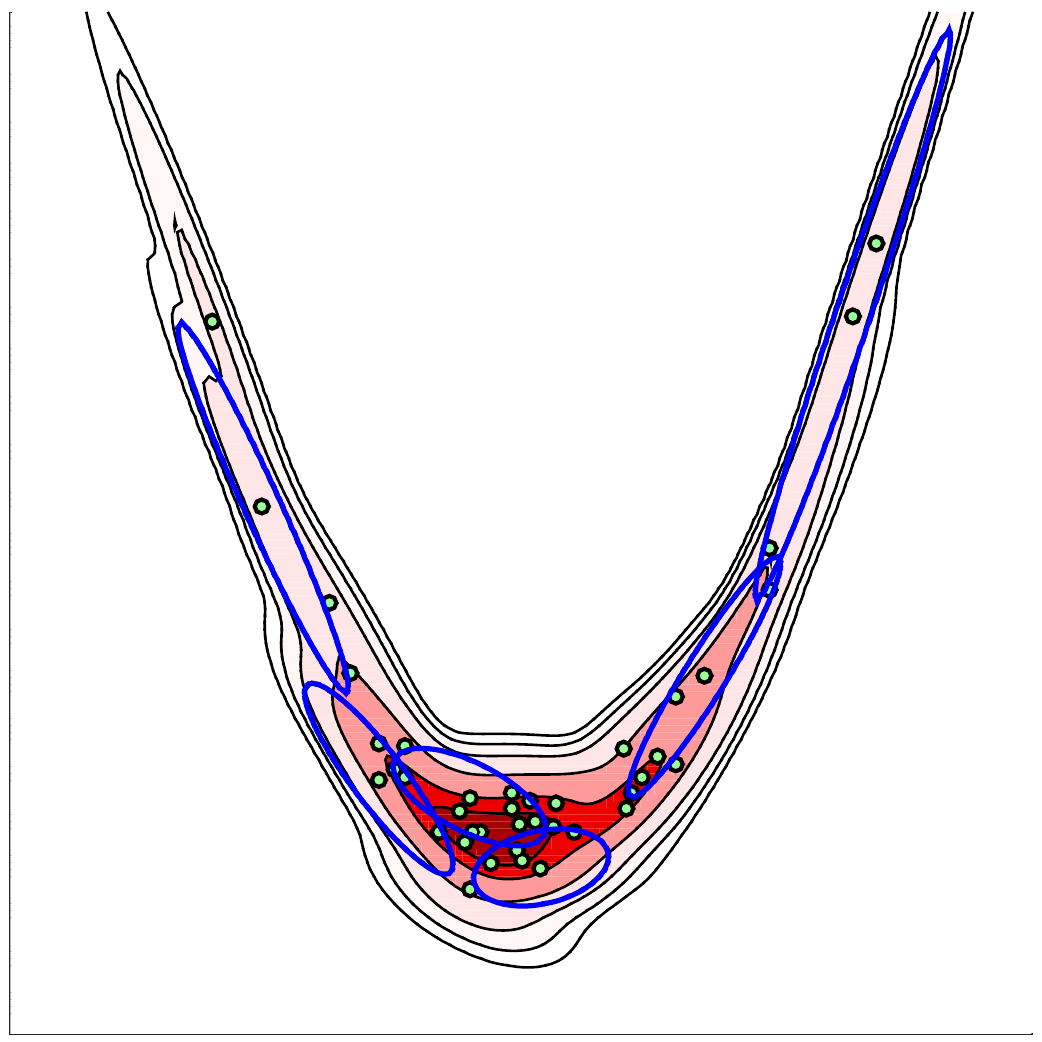}
			\caption{VKDE using \eqref{equ:FinalChoiceBandwidths}}
		\end{subfigure}
        \hspace{0.5cm}
        \begin{subfigure}[b]{0.26\textwidth}
            \centering
			\includegraphics[width=\textwidth]{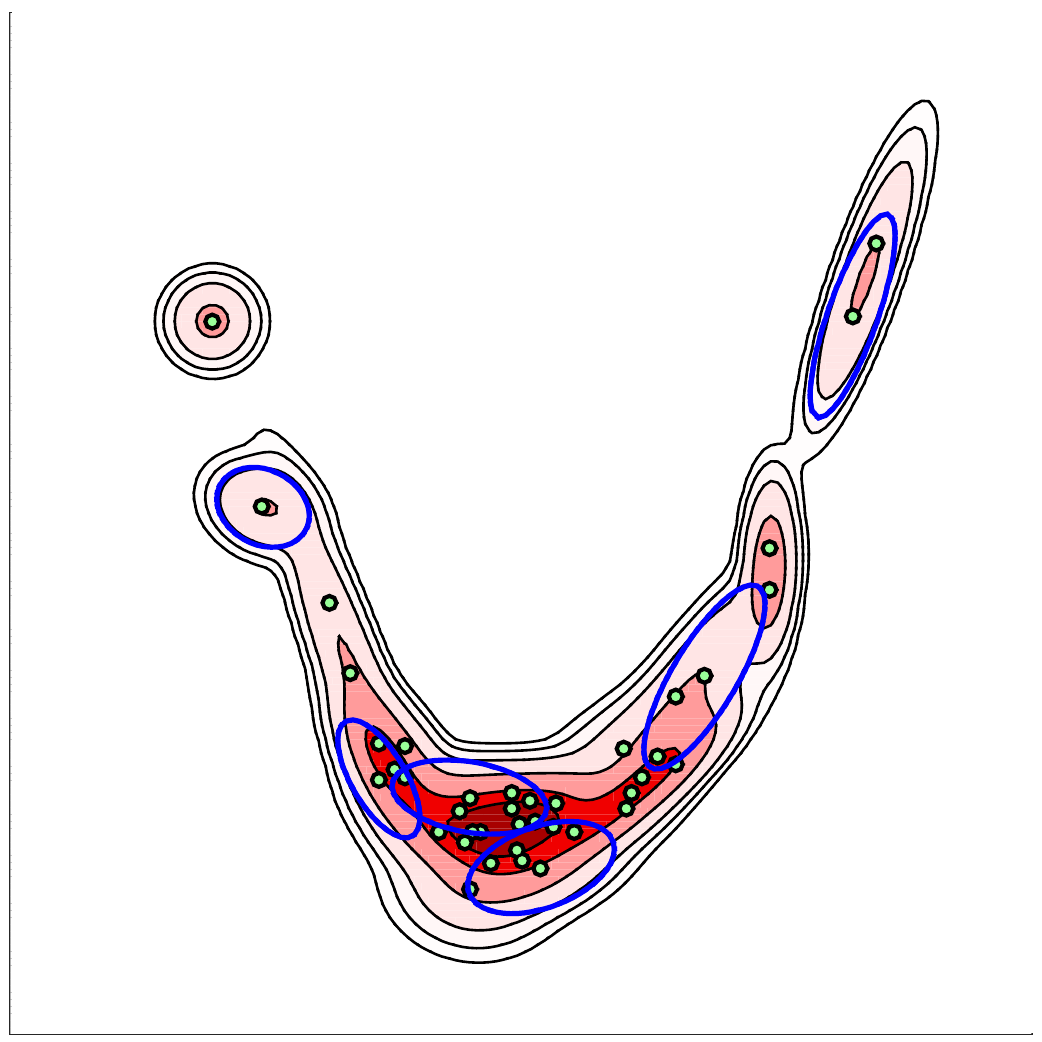}
			\caption{\eqref{equ:densityFixedPointIteration} using \eqref{equ:FinalChoiceBandwidths}}
		\end{subfigure}
        \vfill
        \begin{subfigure}[b]{0.26\textwidth}
            \centering
			\includegraphics[width=\textwidth]{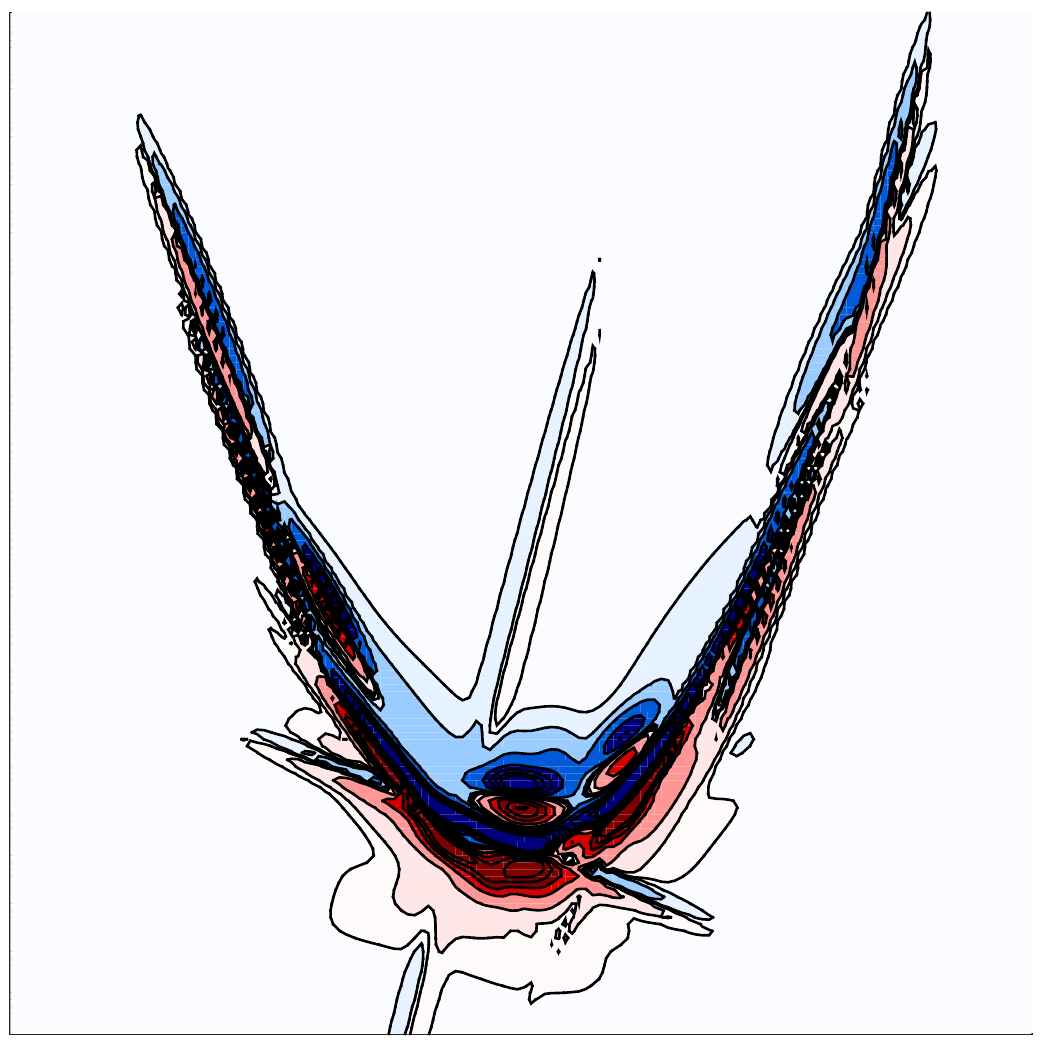}
		\end{subfigure}
        \hspace{0.5cm}
        \begin{subfigure}[b]{0.26\textwidth}
            \centering
			\includegraphics[width=\textwidth]{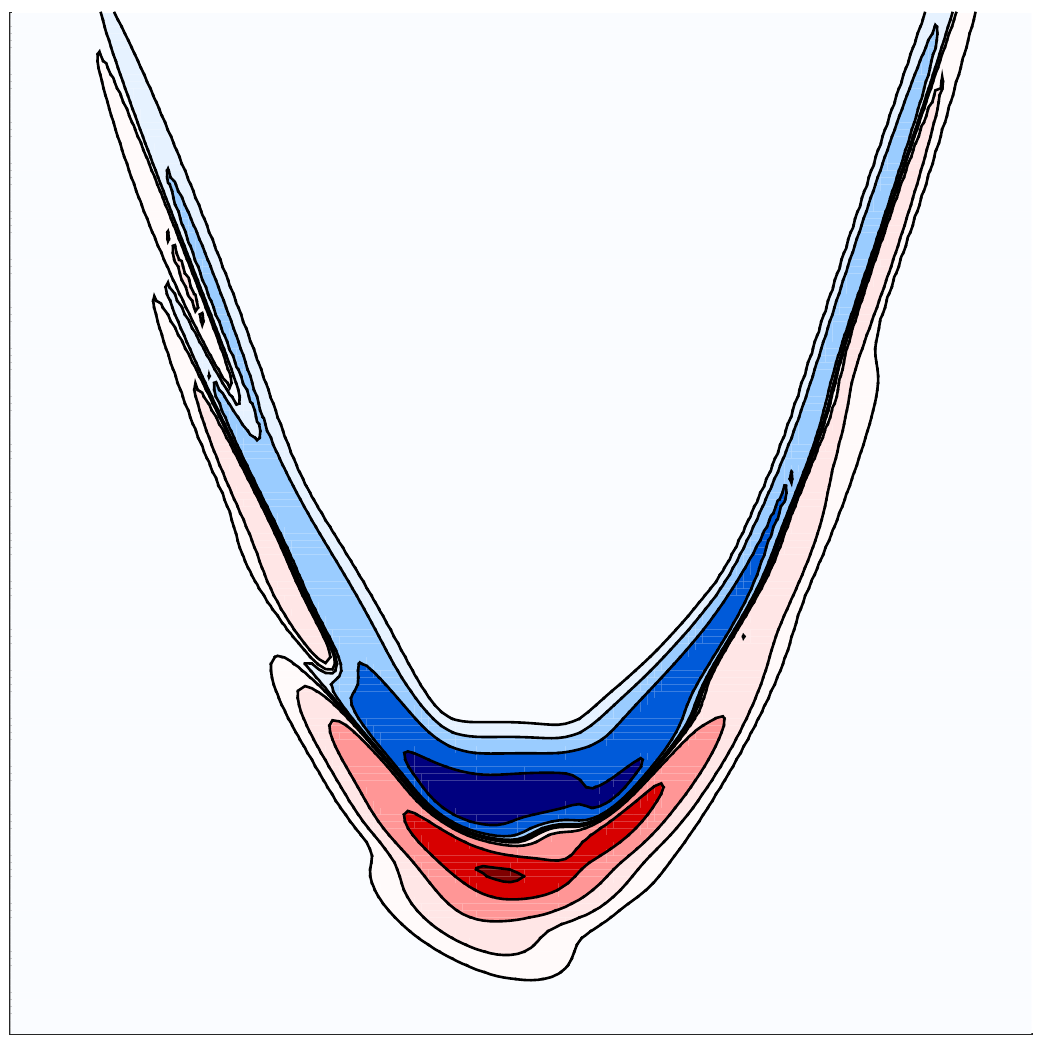}
		\end{subfigure}
        \hspace{0.5cm}
        \begin{subfigure}[b]{0.26\textwidth}
            \centering
			\includegraphics[width=\textwidth]{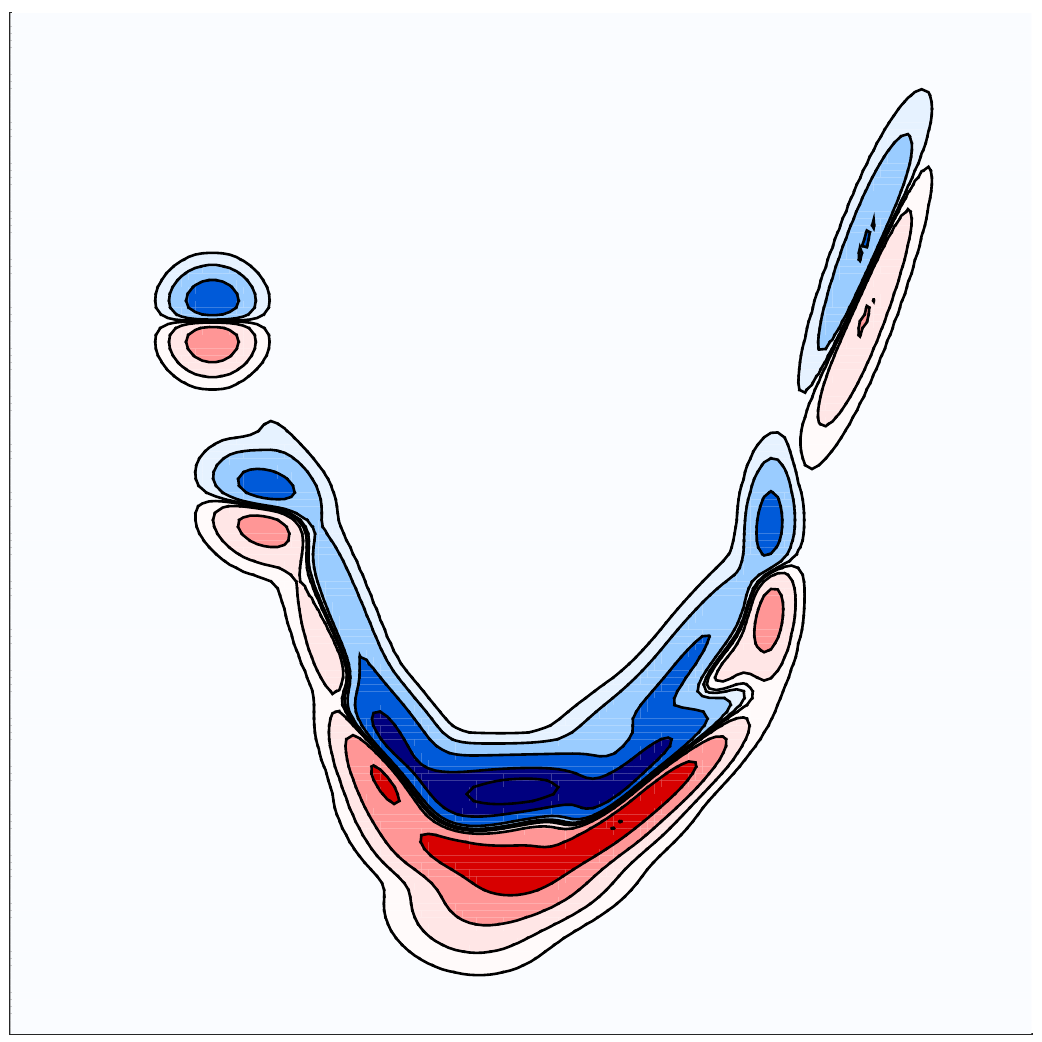}
		\end{subfigure}		
        \caption{Different kernel density estimates $\hat{\rho}$ of the density
        \eqref{equ:bananaDensity} for $N=40$ sample points are plotted together
        with six kernels (ellipses are 80\% contours of six kernels) and their
        derivatives $\partial_{x_2}\hat{\rho}$ below. The overall smoothing in
        (b), (c) and (e) was chosen to minimize the MISE, the one for (f) was chosen
        manually. The advantages of VKDE over the standard KDE is clearly visible. While (c), (d) and (e) were computed using the true density $\rho$ and its derivatives, (f) was computed without this knowledge by
        the fixed point iteration \eqref{equ:densityFixedPointIteration} (10
        iterations were performed).
		}
        \label{fig:ComparisonVKDEs}
\end{figure}

Comparing the performance of the different estimates, we observe
better performance of VKDE in the tails, which is particularly evident from the
derivative plots (the partial derivative in $x_2$-direction is plotted).
Surprisingly, the MISE of the standard KDE is smaller than the one of the VKDE
chosen by \eqref{equ:hLaw} (the sensitivity parameter $\beta = 1/2$ was used). The tails of the latter are just too flat, which
stems from the lack of flexibility in stretching and rotating the kernels.
This does not mean that the overall performance of standard KDE is better, as is
evident from the number of additional modes of the estimates and from the
derivatives, where \eqref{equ:hLaw} outperforms standard KDE.

Parzen's KDE, though provably optimal asymptotically, encounters serious
problems in the finite sample scenario. Though making use of the true density
and its derivatives, it fails to choose flatter kernels in regions of low
density and peaked kernels in regions of high density and the orientation of the
kernels is far from optimal.
In addition, there is at least one serious outlier, which has a disastrous
impact on the overall form of the estimate. Its origin is a positive definite
Hessian $H_x$ of $\rho$ with an eigenvalue close to zero -- the kernel is chosen
absurdly wide in the direction of the corresponding eigenvector.

The bandwidth selector \eqref{equ:FinalChoiceBandwidths} highly outperforms the other
methods. Not only is its MISE considerably lower, it is also the only method
which manages to reproduce the overall form of the density and does not
introduce many additional modes (this is particularly evident from the derivative plots).

Naturally, if the true density $\rho$ and its derivatives are not accessible and
one is forced to apply pilot estimates or fixed point iterations, the
performance of all VKDE estimates suffers. However,
\eqref{equ:FinalChoiceBandwidths} still outperforms the other methods,
especially when comparing the overall form of the approximation (note, that for
the fixed point iteration the proportionality constant $\kappa$ in
\eqref{equ:FinalChoiceBandwidths} was chosen manually and does not minimize the
MISE). Several steps of the fixed point iteration are plotted in Figure
\ref{fig:fpi2d}.

The MISEs of all methods are given in Table \ref{table:MISE}.

\begin{table}[H]
\centering
\begin{small}
\begin{tabular}{|C{2.1cm}||C{1.8cm}|C{1.8cm}|C{1.8cm}|C{1.8cm}|C{1.8cm}|}
\hline
\vspace{0.1cm} Method \vspace{0.1cm} & Standard KDE & VKDE using \eqref{equ:hLaw} & Parzen's VKDE & VKDE using \eqref{equ:FinalChoiceBandwidths} & \eqref{equ:densityFixedPointIteration} using \eqref{equ:FinalChoiceBandwidths}
\\
\hline
\hline
\vspace{0.1cm} MISE of $\hat \rho$ \vspace{0.1cm} & 0.57 & 0.61 & 0.65 &
0.38 & 0.49\\
\hline
\vspace{0.1cm} MISE of $\partial_y\hat\rho$ \vspace{0.1cm} & 0.61 & 0.59 &
1.34 & 0.40 & 0.55\\ 
\hline
\end{tabular}
\end{small}
\caption{MISE of each density estimate $\hat \rho$ and its derivative $\partial_y\hat\rho$.}
\label{table:MISE}
\end{table}

%
%
%

\begin{figure}[H]
        \centering
        \begin{subfigure}[b]{0.32\textwidth}
            \centering
			\includegraphics[width=\textwidth]{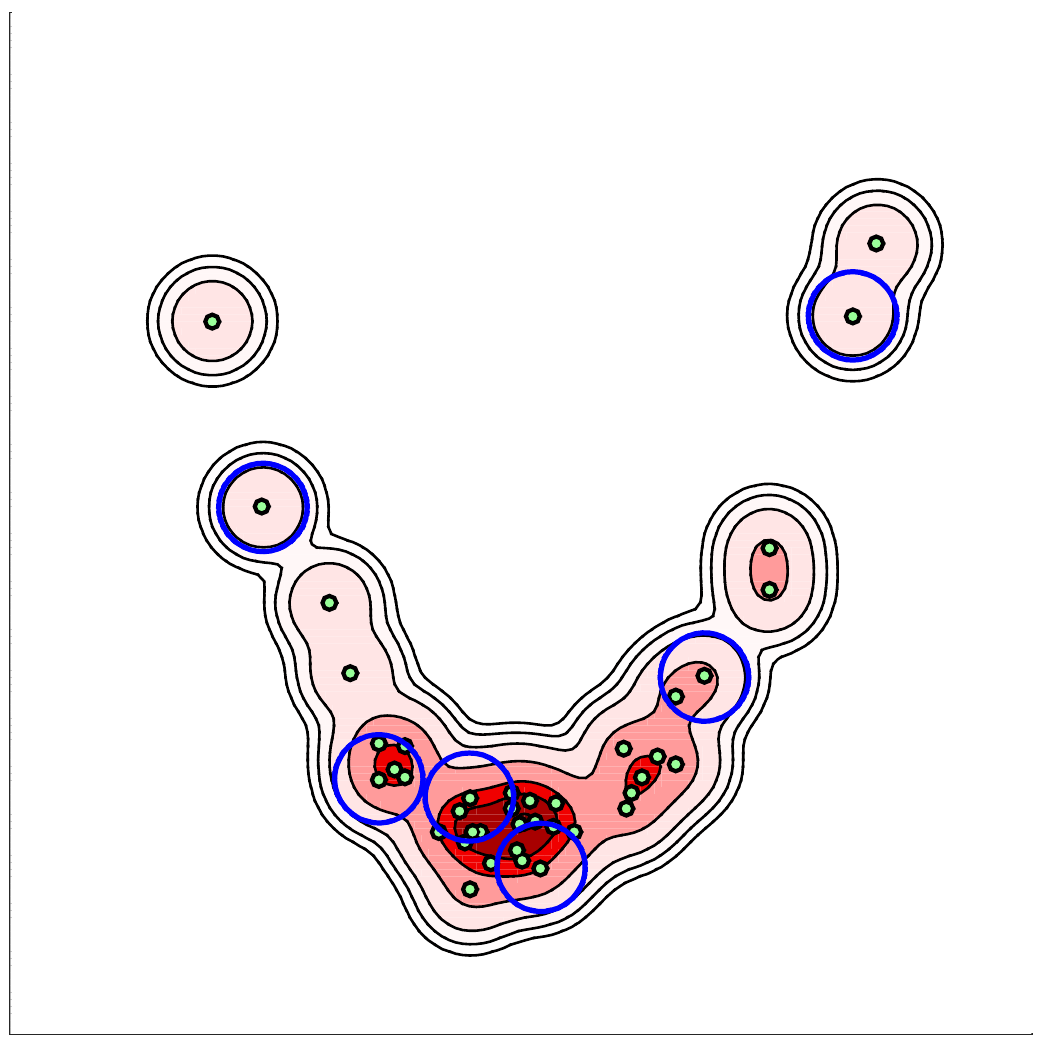}
            \caption{Standard KDE}
            \label{fig:trueDensity}			
        \end{subfigure}
        \hfill
        \begin{subfigure}[b]{0.32\textwidth}
            \centering
			\includegraphics[width=\textwidth]{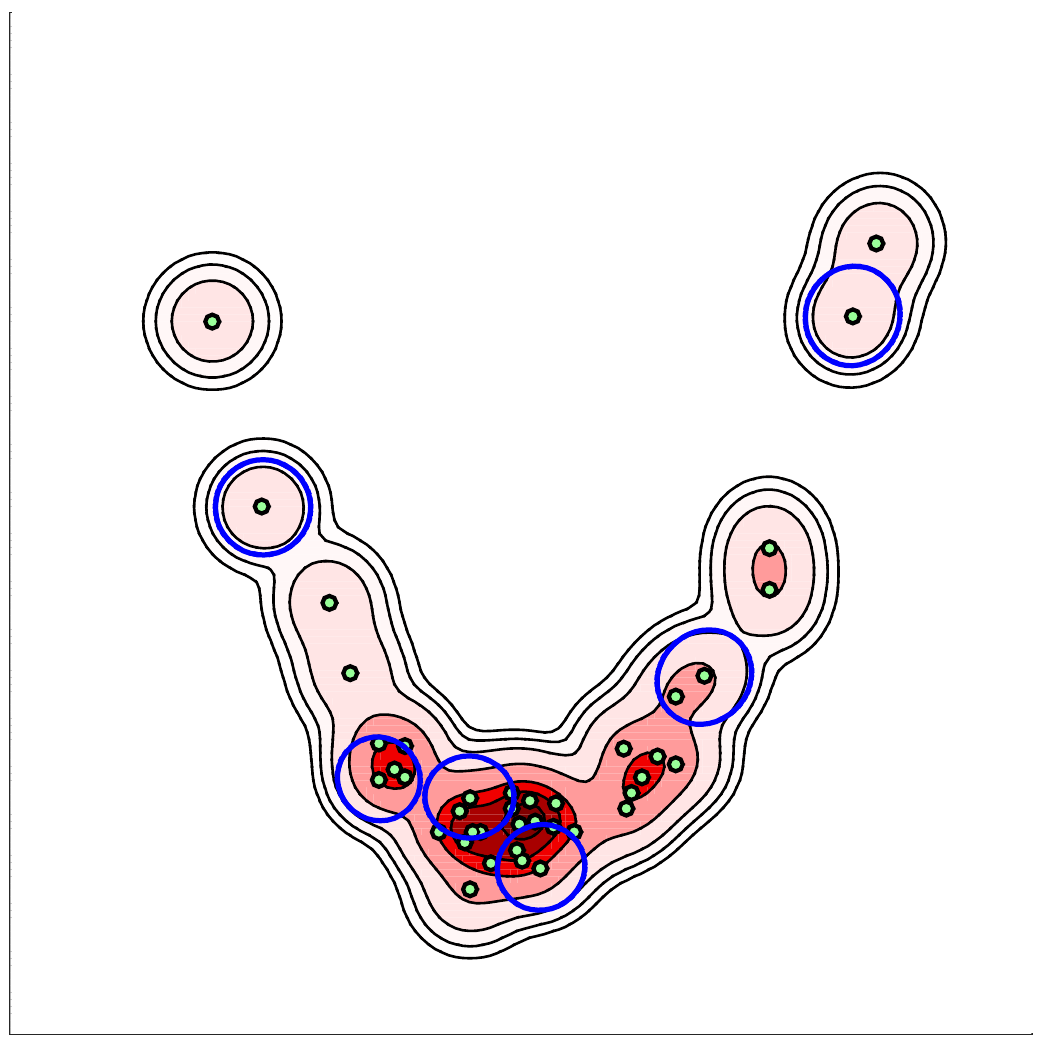}
            \caption{1 iteration step}			
		\end{subfigure}
        \hfill
        \begin{subfigure}[b]{0.32\textwidth}
            \centering
			\includegraphics[width=\textwidth]{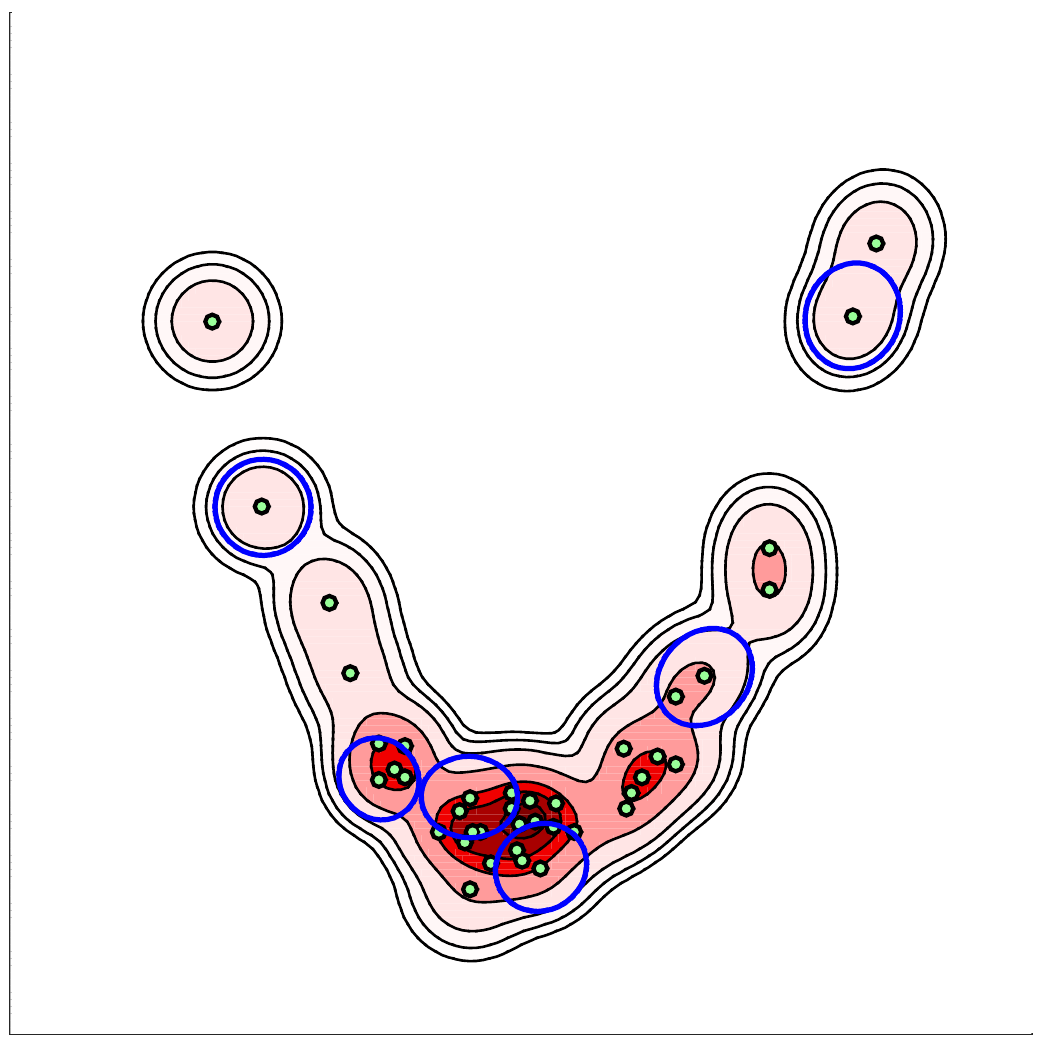}
            \caption{2 iteration steps}			
        \end{subfigure}
        \vfill
        \vspace{0.2cm}
        \begin{subfigure}[b]{0.32\textwidth}
            \centering
			\includegraphics[width=\textwidth]{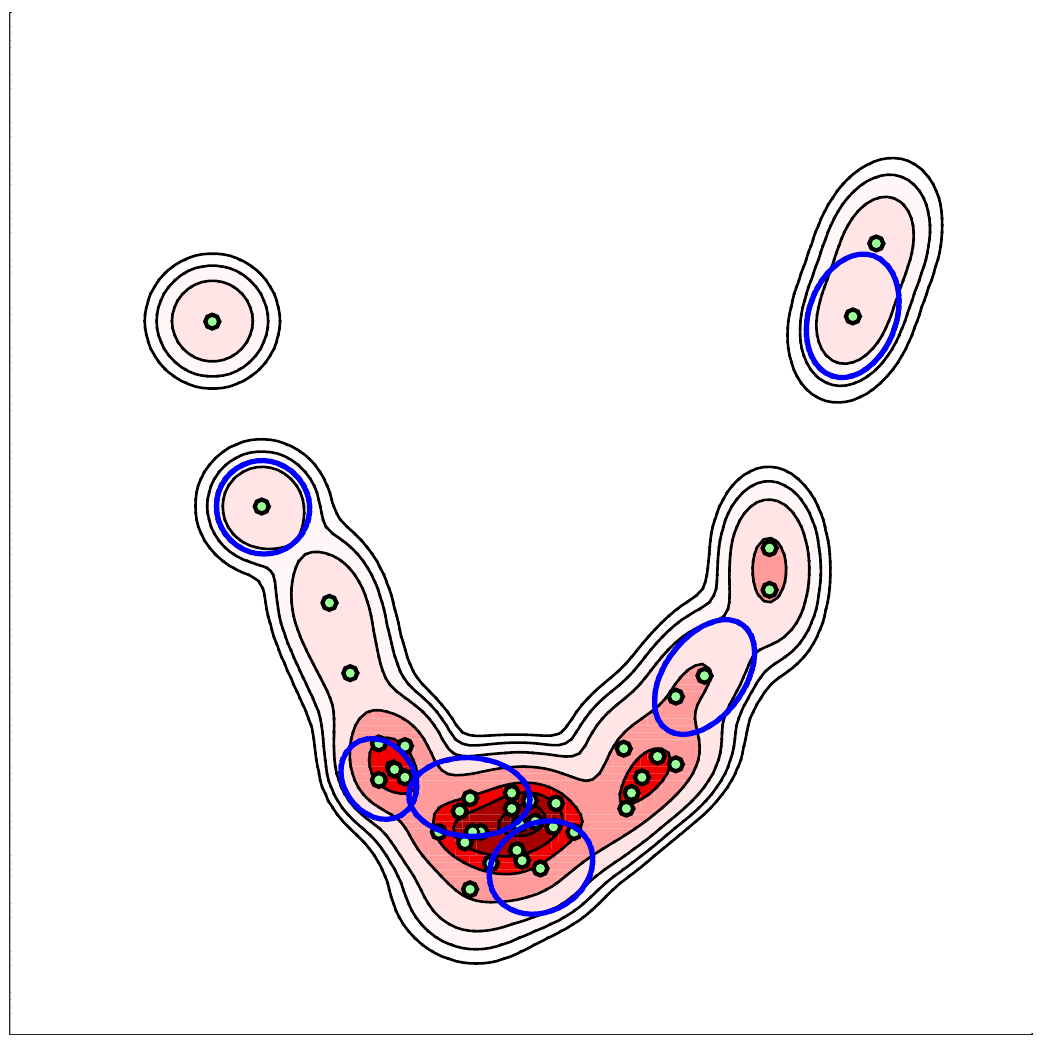}
            \caption{5 iteration steps}						
        \end{subfigure}
        \hfill
        \begin{subfigure}[b]{0.32\textwidth}
            \centering
			\includegraphics[width=\textwidth]{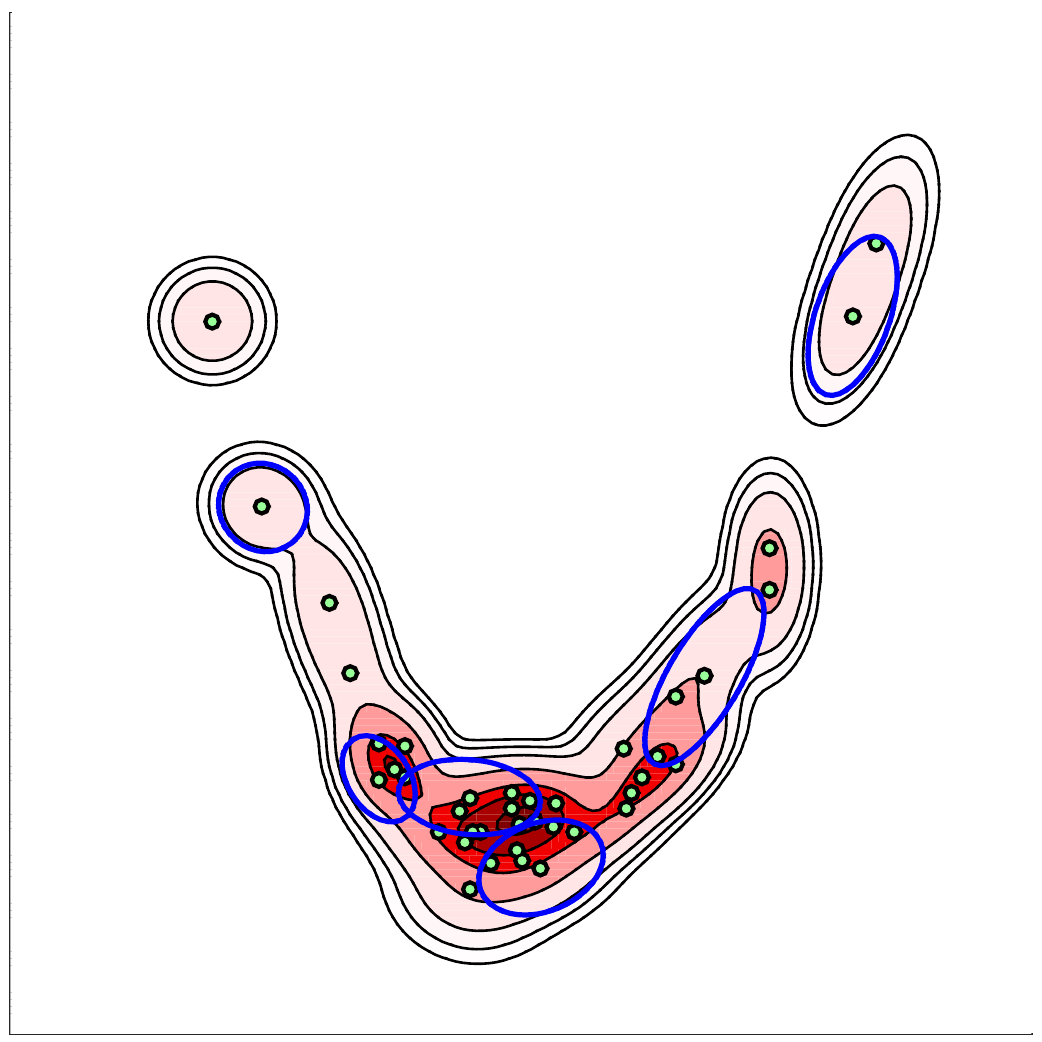}
            \caption{10 iteration steps}						
		\end{subfigure}
        \begin{subfigure}[b]{0.32\textwidth}
            \centering
			\includegraphics[width=\textwidth]{images/fpi20}
            \caption{20 iteration steps}						
		\end{subfigure}
        \caption{The fixed point iteration
        \eqref{equ:densityFixedPointIteration} using bandwidth selector
        \eqref{equ:FinalChoiceBandwidths} starting with
        the standard KDE from Figure \ref{fig:ComparisonVKDEs} (b).
        Ellipses are 80\% contours of six kernels.
        }
		\label{fig:fpi2d}
\end{figure}

%
%
%
%
%
%
%

\subsection{Real Life Example: Earthquake Data}
\label{section:Earthquake}
In his book \cite{zbMATH00907051}, Simonoff\footnote{Courtesy of Jeffrey S. Simonoff, who kindly made the data available on his webpage, \url{http://people.stern.nyu.edu/jsimonof/SmoothMeth/Data/ASCII/quake.dat}}
analyzes the data set `quake.dat' consisting of latitude
and longitude values\footnote{For simplicity, we neglect the curvature of the earth and treat the latitude and longitude values as Cartesian coordinates, which is sufficient for our purposes.} of earthquakes with magnitude at least 5.8 on the Richter scale occurring between January 1964 and February 1986.
We restrict the data to earthquakes occurring in East Asia and the Western Pacific region with magnitude
larger than 6.2 on the Richter scale, thus reducing the number of samples to
145, see Figure \ref{fig:earthquake}(a).
As is evident from Figure \ref{fig:earthquake}, the predictive power of VKDE
using \eqref{equ:FinalChoiceBandwidths} is superior compared to the other two
density estimates.
\begin{figure}[H]
        \centering
        \begin{subfigure}[b]{0.39\textwidth}
            \centering
			\includegraphics[width=\textwidth]{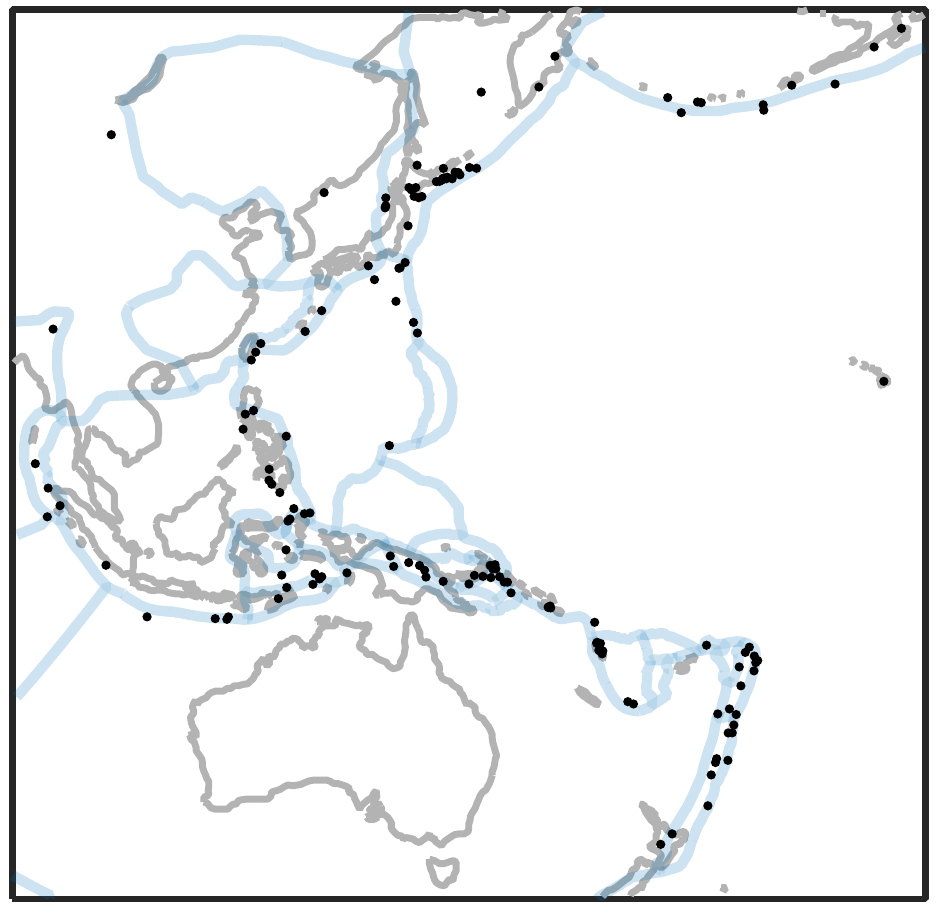}
			\caption{Earthquake data and tectonic plates}
        \end{subfigure}        
        \hspace{1cm}
        \begin{subfigure}[b]{0.39\textwidth}
            \centering
			\includegraphics[width=\textwidth]{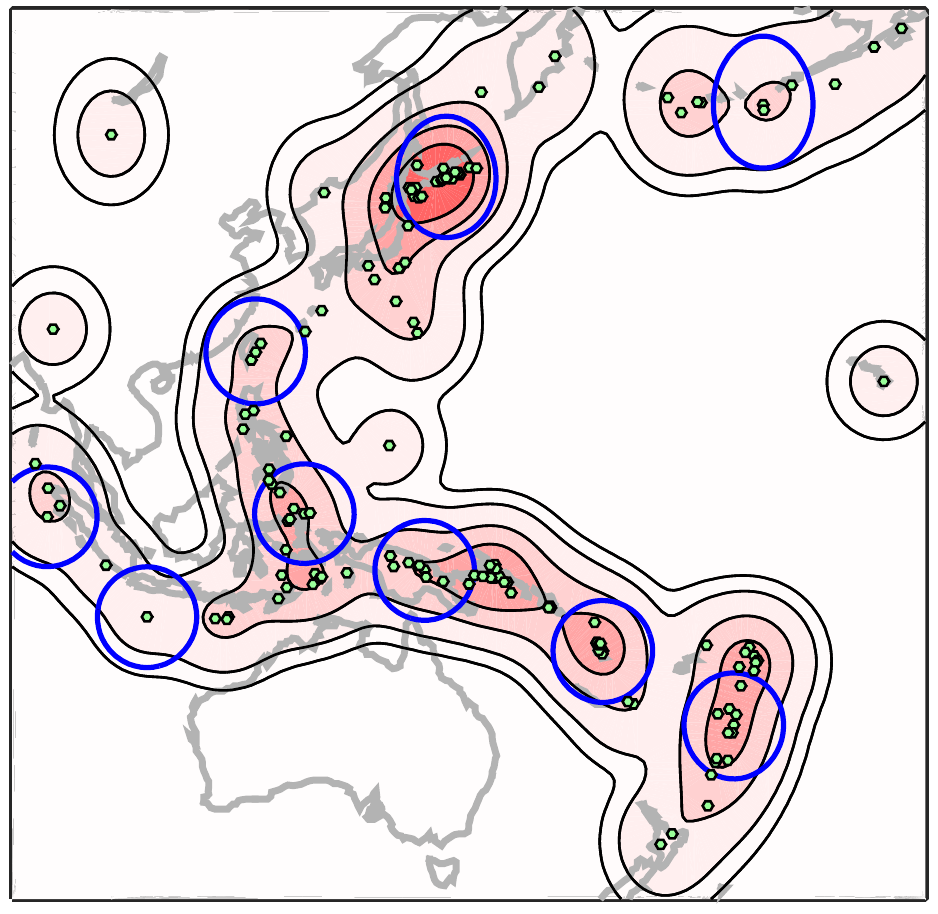}
			\caption{Standard KDE}
		\end{subfigure}
		\vfill
        \vspace{0.1cm}
        \begin{subfigure}[b]{0.39\textwidth}
            \centering
			\includegraphics[width=\textwidth]{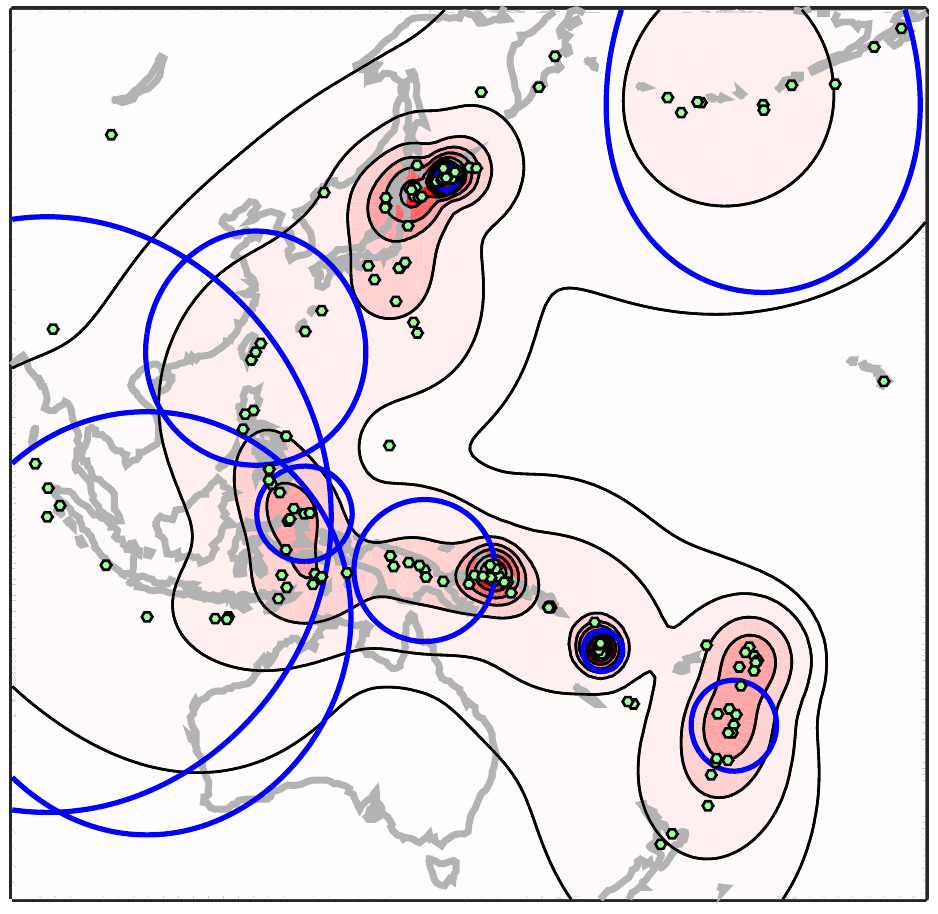}
			\caption{VKDE using
			\eqref{equ:hLaw}, $\beta = 1/2$}
        \end{subfigure}
        \hspace{1cm}
        \begin{subfigure}[b]{0.39\textwidth}
            \centering
			\includegraphics[width=\textwidth]{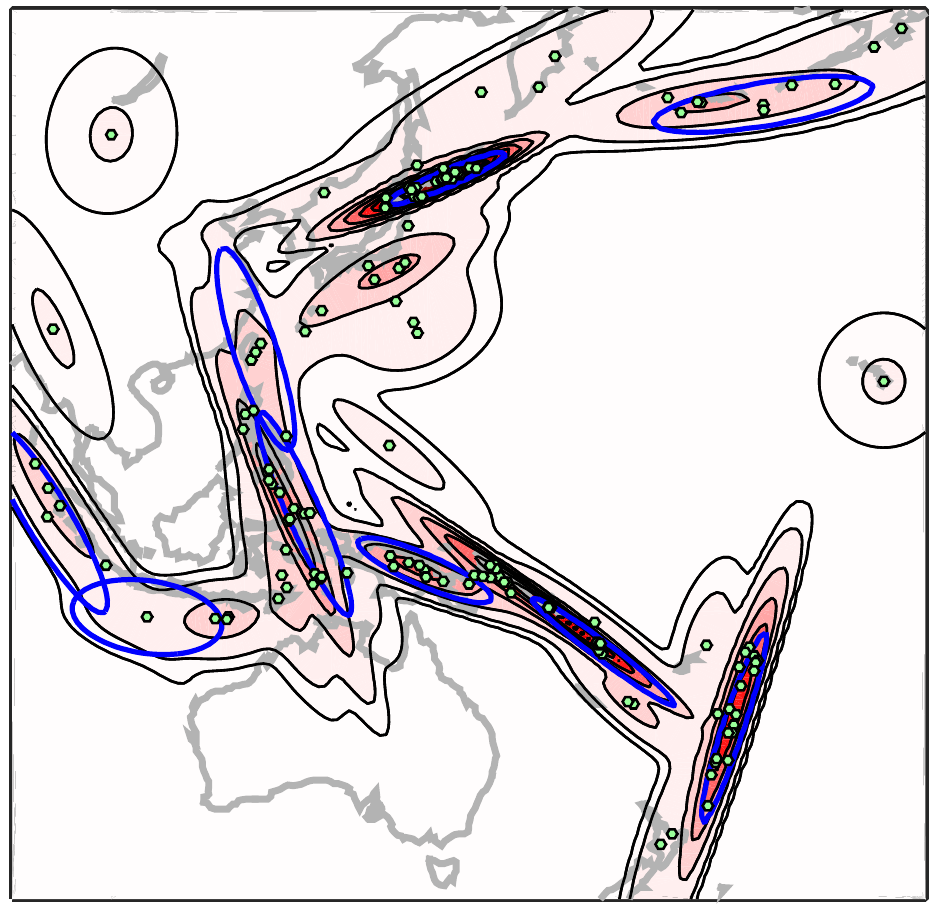}
			\caption{VKDE using
			\eqref{equ:FinalChoiceBandwidths}}
		\end{subfigure}
        \caption{Different KDE approaches applied to the earthquake data
        visualized in (a). For (c) and (d), ten fixed point iteration steps of
        \eqref{equ:densityFixedPointIteration} have been performed, starting
        with the standard KDE from (b). Ellipses are 80\%
        contours of nine kernels.
        }
        \label{fig:earthquake}
\end{figure}

%% file: sections/Conclusion.tex
\section{Conclusion}
\label{section:Conclusion}

We presented an axiomatic approach to VKDE as an alternative to the typical asymptotic analysis. We introduced certain invariance axioms that
we want our density estimate to fulfill and derived a
bandwidth selection rule, which satisfies these axioms.
The axioms and the selection rule are based on the theory of adaptive convolutions and the local variation of a function and allow for kernels that are stretched differently in different directions.
By introducing these criteria, we established a general framework for testing of bandwidth selection rules for plausibility.

The suggested rule \eqref{equ:FinalChoiceBandwidths} was compared to
conventional fixed and variable bandwidth selection rules and performed
considerably better in an artificial as well as in a real life example in
Section \ref{section:NumericalExperiments}.

Though we could find an explicit formula for the fixed point iteration \eqref{equ:densityFixedPointIteration} based on rule \eqref{equ:FinalChoiceBandwidths} in the case of Gaussian kernels, its computation is still very costly, restricting its feasibility to small to moderate sample sizes.
In addition, the convergence properties of said fixed point iteration remains an
open problem as well as the choice of the constant $\kappa$ in
\eqref{equ:FinalChoiceBandwidths}.

We hope that this work motivates the development of further invariant bandwidth selection rules superior to the ones in existence.

%
%
%
%
%
%

%% file: sections/TechnicalDetails.tex

\section{Technical Details for Gaussian Kernels}
\label{section:technicalDetails}

If $K(x) = (2\pi)^{-d/2}\exp(-\norm{x}^2/2)$ is the standard Gaussian kernel, the application of the fixed point iteration \eqref{equ:densityFixedPointIteration} to the choice \eqref{equ:FinalChoiceBandwidths} can be performed without any numerical approximations. The tedious part is the computation of $\mu_{\rho_\mth}$ for
$\rho_\mth$ from \eqref{equ:rhoV}.
Denoting the Gaussian function with mean $y\in\R^d$ and covariance matrix $Q\in\R^{d\times d}$ by
\[
G_{y,Q}(x)
=
\frac{\det Q^{-1/2}}{(2\pi)^{d/2}}\, 
\exp\left[-\tfrac{1}{2}(x-y)^\intercal Q^{-1} (x-y)\right],
\]
and abbreviating $G_{Q} = G_{0,Q}$, we obtain for $y_j\in\R^d$ and covariance matrices $Q_j\in\R^{d\times d}$, $j=1,\dots,3$, by applying standard rules for products, derivatives and convolutions of Gaussians,
\begin{align*}
(G_{y_1,Q_1}\, G_{y_2,Q_2})\ast G_{Q_3}^2(x)
&\ =\ 
\frac{\abs{\det Q_3}^{-1/2}}{(4\pi)^{d/2}}
\, 
G_{Q_1+Q_2}(y_1-y_2)
\, 
G_{y_{12}, Q_{12}+Q_3/2 }(x),
\\[0.2cm]
(\nabla G_{y_1,P_1}\, \nabla G_{y_2,P_2} - G_{y_1,P_1}\, D^2G_{y_2,P_2})
\ast G_{Q_3}^2(x)
&\ =\ 
\frac{\abs{\det Q_3}^{-1/2}}{(4\pi)^{d/2}}
\, 
G_{Q_1+Q_2}(y_1-y_2)
\, 
G_{y_{12}, Q_{12}+Q_3/2 }(x)
\, \times
\\[0.1cm]
\times \, \Big[
(Q_1^{-1} & - Q_2^{-1}) Q_{1234}Q_2^{-1} + \big(\alpha_1(x)-\alpha_2(x)\big) \alpha_1^\intercal(x) + Q_2^{-1}
\Big],
\end{align*}
where
\begin{align*}
&Q_{12} = (Q_1^{-1} + Q_2^{-1})^{-1},\quad&
&y_{12} = Q_{12}(Q_1^{-1}y_1 + Q_2^{-1}y_2),&
\\
&Q_{1234} = (Q_{12}^{-1}+2Q_3^{-1})^{-1},\quad&
&\alpha_j(x) = Q_j^{-1}\big(Q_{1234}(Q_{12}y_{12} + 2Q_3^{-1}x)-y_j\big),\quad j=1,2.&
\\
\end{align*}
Now we only need to plug this into the implicit formula \eqref{equ:choiceMu} for $\mu_{\rho_\mth}$:
\[
\mu_{\rho_\mth}^2 (x)
=
\frac{\sum_{n_1,n_2=1}^{N} (\nabla g_{n_1} \nabla g_{n_2} - g_{n_1} D^2 g_{n_2}) \ast G_{(\lambda \mu_{\rho_\mth})^{-2}(x)}^2}
{(2-\lambda^2) \sum_{n_1,n_2=1}^{N} (g_{n_1}g_{n_2})\ast G_{(\lambda \mu_{\rho_\mth})^{-2}(x)}^2}\, (x),
\qquad
g_n := G_{y_n,\, (h_n h_n^{\intercal})^{-1}},
\]
which can be solved by yet another fixed point iteration.

%% file: VariableKDE.bbl
\begin{thebibliography}{10}

\bibitem{abramson1982arbitrariness}
I.~S. Abramson.
\newblock Arbitrariness of the pilot estimator in adaptive kernel methods.
\newblock {\em Journal of Multivariate Analysis}, 12(4):562--567, 1982.

\bibitem{zbMATH03800791}
I.~S. {Abramson}.
\newblock {On bandwidth variation in kernel estimates -- a square root law}.
\newblock {\em {Ann. Stat.}}, 10:1217--1223, 1982.

\bibitem{botev2010kernel}
Z.~I. Botev, J.~F. Grotowski, D.~P. Kroese, et~al.
\newblock Kernel density estimation via diffusion.
\newblock {\em The Annals of Statistics}, 38(5):2916--2957, 2010.

\bibitem{zbMATH03591205}
L.~{Breiman}, W.~{Meisel}, and E.~{Purcell}.
\newblock {Variable kernel estimates of multivariate densities.}
\newblock {\em {Technometrics}}, 19:135--144, 1977.

\bibitem{jones1990variable}
M.~Jones.
\newblock Variable kernel density estimates and variable kernel density
  estimates.
\newblock {\em Australian \& New Zealand Journal of Statistics},
  32(3):361--371, 1990.

\bibitem{jones1996brief}
M.~C. Jones, J.~S. Marron, and S.~J. Sheather.
\newblock A brief survey of bandwidth selection for density estimation.
\newblock {\em Journal of the American Statistical Association},
  91(433):401--407, 1996.

\bibitem{2018arXiv180500703K}
I.~{Klebanov}.
\newblock {Adaptive Convolutions}.
\newblock {\em ArXiv e-prints}, 2018.

\bibitem{loader1999bandwidth}
C.~R. Loader.
\newblock Bandwidth selection: classical or plug-in?
\newblock {\em Annals of Statistics}, pages 415--438, 1999.

\bibitem{zbMATH03188880}
E.~{Parzen}.
\newblock {On estimation of a probability density function and mode.}
\newblock {\em {Ann. Math. Stat.}}, 33:1065--1076, 1962.

\bibitem{rosenblatt1956remarks}
M.~Rosenblatt et~al.
\newblock Remarks on some nonparametric estimates of a density function.
\newblock {\em The Annals of Mathematical Statistics}, 27(3):832--837, 1956.

\bibitem{scott2015multivariate}
D.~W. Scott.
\newblock {\em Multivariate density estimation: theory, practice, and
  visualization}.
\newblock John Wiley \& Sons, 2015.

\bibitem{silverman1986density}
B.~W. Silverman.
\newblock {\em Density estimation for statistics and data analysis}, volume~26.
\newblock CRC press, 1986.

\bibitem{zbMATH00907051}
J.~S. {Simonoff}.
\newblock {\em {Smoothing methods in statistics.}}
\newblock New York, NY: Springer, 1996.

\bibitem{terrell1992variable}
G.~R. Terrell and D.~W. Scott.
\newblock Variable kernel density estimation.
\newblock {\em The Annals of Statistics}, pages 1236--1265, 1992.

\bibitem{wand1994kernel}
M.~P. Wand and M.~C. Jones.
\newblock {\em Kernel smoothing}.
\newblock Crc Press, 1994.

\end{thebibliography}
